\documentclass[10pt]{amsart}
\usepackage{amssymb}
\usepackage{bm}
\usepackage{graphicx}
\usepackage[centertags]{amsmath}
\usepackage{amsfonts}
\usepackage{amsthm}
\usepackage{amsbsy}
\usepackage{mathtools}
\usepackage{mathrsfs}
\usepackage{cases}
\usepackage[all]{xy}
\usepackage{hyperref}
\newtheorem{thm}{Theorem}[section]
\newtheorem{cor}[thm]{Corollary}
\newtheorem{lem}[thm]{Lemma}

\newtheorem{prop}[thm]{Proposition}
\newtheorem{claim}[thm]{Claim}

\newtheorem{defn}[thm]{Definition}
\theoremstyle{definition}

\newcommand{\rr}{\mathbb{R}}
\newcommand{\nn}{\mathbb{N}}
\newcommand{\ee}{\varepsilon}

\newcommand{\meg}{\geqslant}
\newcommand{\mik}{\leqslant}
\newcommand{\ave}{\mathbb{E}}

\newcommand{\calb}{\mathcal{B}}

\newcommand{\calg}{\mathcal{G}}
\newcommand{\calh}{\mathcal{H}}
\newcommand{\cals}{\mathcal{S}}

\newcommand{\bbx}{\boldsymbol{X}}
\newcommand{\bmu}{\boldsymbol{\mu}}
\newcommand{\bbs}{\boldsymbol{\Sigma}}
\newcommand{\bnu}{\boldsymbol{\nu}}
\newcommand{\bpsi}{\boldsymbol{\psi}}

\newcommand{\bx}{\mathbf{x}}

\newcommand{\bfu}{\mathbf{f}}

\begin{document}

\title{$L_p$ regular sparse hypergraphs: box norms}

\author{Pandelis Dodos, Vassilis Kanellopoulos and Thodoris Karageorgos}

\address{Department of Mathematics, University of Athens, Panepistimiopolis 157 84, Athens, Greece}
\email{pdodos@math.uoa.gr}

\address{National Technical University of Athens, Faculty of Applied Sciences,
Department of Mathematics, Zografou Campus, 157 80, Athens, Greece}
\email{bkanel@math.ntua.gr}

\address{Department of Mathematics, University of Athens, Panepistimiopolis 157 84, Athens, Greece}
\email{tkarageo@math.uoa.gr}

\thanks{2010 \textit{Mathematics Subject Classification}: 05C35, 05C65,  46B20, 46B25.}
\thanks{\textit{Key words}: sparse hypergraphs, Gowers box norms.}


\begin{abstract}
We consider some variants of the Gowers box norms, introduced by Hatami, and show their relevance in the context of sparse hypergraphs.
Our main results are the following. Firstly, we prove a generalized von Neumann theorem for $L_p$ graphons. Secondly, we give natural examples
of pseudorandom families, that is, sparse weighted uniform hypergraphs which satisfy relative versions of the counting and removal lemmas.
\end{abstract}

\maketitle


\section{Introduction}

\numberwithin{equation}{section}

\subsection{Overview}

Let $\langle (X_i,\Sigma_i,\mu_i):i\in e\rangle$ be a nonempty finite family of probability spaces and let
$(\bbx_e,\bbs_e,\bmu_e)$ denote their product. Recall that the \textit{box norm} of a random variable $f\colon \bbx_e\to\rr$ is the quantity
\begin{equation} \label{e1.1}
\|f\|_{\square(\bbx_e)}\coloneqq \ave\Big[ \prod_{\omega\in\{0,1\}^e} f(\bx^{(\omega)}_e)\, \Big|\, \bx^{(0)}_e,\bx^{(1)}_e\in\bbx_e\Big]^{1/2^{|e|}}.
\end{equation}
(For unexplained notation see Subsection 1.2 below.) These norms were introduced by Gowers \cite{Go1,Go3} and are important tools
in arithmetic and extremal combinatorics. There are some slight variants of the box norms which first appeared\footnote{Actually,
the framework in \cite{Ha1,Ha2} is more general and includes several other variants of \eqref{e1.1}.} in \cite{Ha1,Ha2}: for every even
integer $\ell\meg 2$ we define the \textit{$\ell$-box norm} of $f\colon \bbx_e\to\rr$ by the rule
\begin{equation} \label{e1.2}
\|f\|_{\square_\ell(\bbx_e)}\coloneqq \ave\Big[ \prod_{\omega\in\{0,\dots,\ell-1\}^e} f(\bx^{(\omega)}_e)\, \Big|\,
\bx^{(0)}_e,\dots,\bx^{(\ell-1)}_e\in\bbx_e\Big]^{1/\ell^{|e|}}.
\end{equation}
Clearly, the $\square_2(\bbx_e)$-norm coincides with the $\square(\bbx_e)$-norm. As the parameter $\ell$ increases, the
quantity $\|f\|_{\square_\ell(\bbx_e)}$ also increases and depends on the integrability properties of $f$. In particular,
for bounded functions all these norms are essentially equivalent (see \cite[Proposition A.1]{DK}), but for unbounded functions
they behave quite differently.

The starting point of this paper is the observation that the $\ell$-box norms can serve as the proper higher-complexity\footnote{Note,
here, that if $s, r$ are positive integers with $s>r$, then there is no analogue of the Gowers $U^s$-norm for $r$-uniform hypergraphs.}
analogues of the box norms in the context of sparse hypergraphs and related structures. A strong indication which supports this point
of view is that the Gowers--Cauchy--Schwarz inequality also holds for the $\square_\ell(\bbx_e)$-norms. This fact together with several
elementary properties are discussed in Section 2.

The rest of this paper is devoted to the proof of our main results which use the $\ell$-box norms in an essentially way
(further examples showing the relevance of these norms are given in \cite{DK}). In Section~3 we present a version of the
\textit{generalized von Neumann theorem} for $L_p$ graphons (Theorem \ref{t3.1} in the main text); as we shall discuss in more
detail in Section 3, the main point in this result is that it can be applied to $L_p$ graphons for any $p>1$. The second part of
this paper deals with \emph{pseudorandom families} \cite{DKK}, a class of sparse weighted uniform hypergraphs whose most important
feature is that they satisfy relative versions of the counting and removal lemmas. Their definition is recalled in Section 4,
but for a more complete discussion of their properties we refer the reader to \cite{DKK}. We present two different types of examples
of pseudorandom families (see Theorems~\ref{t4.2}~and~\ref{t4.3} in the main text). They both can be seen as deviations (in an $L_p$-sense)
of hypergraphs which satisfy the linear forms condition, a well-known pseudorandomness condition originating from \cite{GT1}.

\subsection{Background material}

Our general notation and terminology is standard. By $\nn=\{0,1,\dots\}$ we denote the set of all natural numbers. As usual,
for every positive integer $n$ we set $[n]\coloneqq \{1,\dots,n\}$. If $f$ is an integrable real-valued random variable defined on
a probability space $(X,\Sigma,\mu)$, then by $\ave[f(x)\, |\, x\in X]$ we shall denote the expected value of $f$; if the sample space
$X$ is understood from the context, then the expected value of $f$ will be denoted simply by $\ave[f]$. All necessary background from
probability theory needed in this paper can be found, e.g., in \cite{Bi}.

As we have noted, the box norms and their variants are associated with finite products of probability spaces. It is more convenient,
however, to work with the following more general structures.
\begin{defn}[\cite{Tao1}] \label{d1.1}
A \emph{hypergraph system} is a triple
\begin{equation} \label{e1.3}
\mathscr{H}=(n,\langle(X_i,\Sigma_i,\mu_i):i\in [n]\rangle,\calh)
\end{equation}
where $n$ is a positive integer, $\langle(X_i,\Sigma_i,\mu_i): i\in [n]\rangle$ is a finite sequence of probability spaces and $\calh$
is a hypergraph on $[n]$. If\, $\calh$ is $r$-uniform, then $\mathscr{H}$ will be called an \emph{$r$-uniform hypergraph system}.
\end{defn}
For every hypergraph system $\mathscr{H}=(n,\langle(X_i,\Sigma_i,\mu_i):i\in [n]\rangle,\calh)$ by $(\bbx,\bbs,\bmu)$ we shall denote
the product of the spaces $\langle (X_i,\Sigma_i,\mu_i): i\in [n]\rangle$. More generally, let $e\subseteq [n]$ be nonempty and let
$(\bbx_e,\bbs_e,\bmu_e)$ denote the product of the spaces $\langle (X_i,\Sigma_i,\mu_i): i\in e\rangle$. (By convention, we set
$\bbx_{\emptyset}$ to be the empty set.) The $\sigma$-algebra $\bbs_e$ is not comparable with $\bbs$, but it can be ``lifted"
to $\bbx$ by setting
\begin{equation} \label{e1.4}
\calb_e=\big\{ \pi^{-1}_e(\mathbf{A}): \mathbf{A}\in\bbs_e\big\}
\end{equation}
where $\pi_e\colon \bbx\to\bbx_e$ is the natural projection. Observe that if $f\in L_1(\bbx,\calb_e,\bmu)$, then
there exists a unique random variable $\bfu\in L_1(\bbx_e,\bbs_e,\bmu_e)$ such that
\begin{equation} \label{e1.5}
f=\bfu\circ \pi_e
\end{equation}
and note that the map $L_1(\bbx,\calb_e,\bmu)\ni f\mapsto \bfu\in L_1(\bbx_e,\bbs_e,\bmu_e)$ is a linear isometry.
We will also deal with products of the space $(\bbx_e,\bbs_e,\bmu_e)$. Specifically, let $\ell\in\nn$ with $\ell\meg 2$.
For every $\bx_e^{(0)}=(x_i^{(0)})_{i\in e},\dots,\bx_e^{(\ell-1)}=(x_i^{(\ell-1 )})_{i\in e}$ in $\bbx_e$ and every
$\omega=(\omega_i)_{i\in e}\in\{0,\ldots,\ell-1\}^e$ we set
\begin{equation} \label{e1.6}
\bx_e^{(\omega)}=(x_i^{(\omega_i)})_{i \in e}\in\bbx_e.
\end{equation}
Notice that if $\omega=m^e$ for some $m\in\{0,\dots,\ell-1\}$ (that is, $\omega=(\omega_i)_{i\in e}$ with $\omega_i=m$ for every $i\in e$),
then $\bx_e^{(\omega)}=\bx_e^{(m)}$.


\section{$\ell$-box norms}

\numberwithin{equation}{section}

In this section we will present several elementary properties of the $\ell$-box norms. We will follow the exposition in \cite[~Appendix~B]{GT2}
quite closely. In what follows, let $\mathscr{H}=(n,\langle(X_i,\Sigma_i,\mu_i):i\in [n]\rangle,\calh)$ denote a hypergraph system.

\subsection{Basic properties}

Let $e\subseteq [n]$ be nonempty and let $\ell\meg 2$ be an even integer. Also let $f\in L_1(\bbx_e,\boldsymbol{\Sigma}_e,\bmu_e)$.
We first observe that the $\ell$-box norm of $f$ can be recursively defined as follows. If $e=\{j\}$ is a singleton, then by \eqref{e1.2} we have
\begin{equation} \label{e2.1}
\| f\|_{\square_\ell(\bbx_e)}=\ave\Big[\prod_{\omega=0}^{\ell-1}f(x_j^{(\omega)})\,\Big|\, x^{(0)}_j,\dots, x^{(\ell-1)}_j\in X_j\Big]^{1/\ell^{|e|}}
\!\!\!\!= \big(\ave[f]^\ell\big)^{1/\ell}\!\!=|\ave[f]|.
\end{equation}
On the other hand, if $|e|\meg 2$, then for every $j\in e$ we have
\begin{equation} \label{e2.2}
\|f\|_{\square_\ell(\bbx_e)}= \ave\Big[\big\|\prod_{\omega=0}^{\ell-1}f(\,\cdot\, ,x^{(\omega)}_j)\big\|^{\ell^{|e|-1}}_{\square_\ell(\bbx_{e\setminus\{j\}})}
\,\Big|\, x_j^{(0)},\dots, x_j^{(\ell-1)}\in X_j\Big]^{1/\ell^{|e|}}.
\end{equation}
We have the following proposition.
\begin{prop} \label{p2.1}
Let $e\subseteq [n]$ be nonempty and let $\ell\meg 2$ be an even integer.
\begin{enumerate}
\item[(a)] \emph{$($Gowers--Cauchy--Schwarz inequality$)$} For every $\omega\in\{0,\dots,\ell-1\}^e$
let $f_\omega~\in~L_1(\bbx_e,\boldsymbol{\Sigma}_e,\bmu_e)$. Then we have
\begin{equation} \label{e2.3}
\Big| \ave\Big[\prod_{\omega\in\{0,\dots,\ell-1\}^e} f_\omega(\bx_e^{(\omega)})\,\Big|\, \bx_e^{(0)},\dots,\bx_e^{(\ell-1)}\in\bbx_e\Big]\Big|
\mik \!\! \prod_{\omega\in \{0,\dots,\ell-1\}^e}\!\!\!\|f_\omega\|_{\square_\ell(\bbx_e)}.
\end{equation}
\item[(b)] Let $f\in L_1(\bbx_e,\boldsymbol{\Sigma}_e,\bmu_e)$. Then we have $|\ave[f]|\mik\|f\|_{\square_\ell(\bbx_e)}$. Moreover, if\,
$\ell_1\mik \ell_2$ are even positive integers, then $\|f\|_{\square_{\ell_1}(\bbx_e)}\mik \|f\|_{\square_{\ell_2}(\bbx_e)}$.
\item[(c)] If\, $|e|\meg 2$, then $\|\cdot\|_{\square_\ell(\bbx_e)}$ is a norm on the vector subspace of $L_1(\bbx_e,\boldsymbol{\Sigma}_e,\bmu_e)$
consisting of all $f\in L_1(\bbx_e,\boldsymbol{\Sigma}_e,\bmu_e)$ with $\|f\|_{\square_{\ell}(\bbx_e)}<\infty$.
\item[(d)] Let $1<p\mik\infty$ and let $q$ denote the conjugate exponent of $p$. Assume that $\ell\meg q$ and that $e=\{i,j\}$ is a doubleton.
Then for every $f\in L_1(\bbx_e, \boldsymbol{\Sigma}_e,\bmu_e)$, every $u\in L_p(X_i,\Sigma_i,\mu_i)$ and every $v\in L_p(X_j,\Sigma_j, \mu_j)$ we have
\begin{equation} \label{e2.4}
|\ave[f(x_i,x_j) u(x_i) v(x_j)\, | \, x_i \in X_i, x_j \in X_j]| \mik \|f\|_{\square_{\ell}(\bbx_e)}\, \|u\|_{L_p} \|v\|_{L_p}.
\end{equation}
\end{enumerate}
\end{prop}
\begin{proof}[Proof of Proposition \emph{\ref{p2.1}}]
(a) We follow the proof from \cite[~Lemma B.2]{GT2} which proceeds by induction on the cardinality of $e$.
The case ``$|e|=1$" is straightforward, and so let $r\meg 2$ and assume that the result has been proved for every
$e'\subseteq [n]$ with $1\mik |e'|\mik r-1$. Let $e\subseteq [n]$ with $|e|=r$ be arbitrary. Fix $j\in e$, set $e'=e\setminus\{j\}$
and for every $\omega\in\{0,\dots,\ell-1\}^e$ let $f_\omega~\in~L_1(\bbx_e,\boldsymbol{\Sigma}_e,\bmu_e)$. Moreover, for every
$\omega_j\in \{0,\dots,\ell-1\}$ define $G_{\omega_j}\colon\bbx_{e'}^\ell\to \rr$ by
\begin{equation} \label{e2.5}
G_{\omega_j}(\bx_{e'}^{(0)},\dots,\bx_{e'}^{(\ell-1)})=
\ave\Big[\prod_{\omega_{e'}\in\{0,\dots,\ell-1\}^{e'}}\!\!\! f_{(\omega_{e'}, \omega_j)}(\bx^{(\omega_{e'})}_{e'}, x_j)\,\Big|\, x_j\in X_j\Big]
\end{equation}
where $(\omega_{e'}, \omega_j)$ is the unique element $\omega$ of $\{0,\dots,\ell-1\}^e$ such that $\omega(j)=\omega_j$ and
$\omega(i)=\omega_{e'}(i)$ for every $i\in e'$. Observe that
\[ \Big| \ave\Big[\prod_{\omega\in\{0,\dots,\ell-1\}^e} \!\!\! f_\omega(\bx_e^{(\omega)})\,\Big|\, \bx_e^{(0)},\dots,\bx_e^{(\ell-1)}\in\bbx_e\Big]\Big|
= \Big|\ave\Big[ \prod_{\omega_j=0}^{\ell-1} G_{\omega_j}\Big]\Big| \]
and, by H\"{o}lder's inequality\footnote{Here, and in the rest of the proof, we use the following form of H\"{o}lder's inequality:
if $(X,\Sigma,\mu)$ is a probability space, then for every integer $k\meg 2$, every $p_1,\dots,p_k>1$ with $\sum_{i=1}^k 1/p_i=1$, and every
$f_1,\dots,f_k\colon X\to \rr$ with $f_i\in L_{p_i}(X,\Sigma,\mu)$ for all $i\in [k]$, we have
\[ \big| \ave \big[ \prod_{i=1}^{k} f_i \big] \big| \mik \prod_{i=1}^{k} \|f_i\|_{L_{p_i}}.\] },
$|\ave\big[\prod_{\omega_j=0}^{\ell-1}G_{\omega_j}\big]|\mik\prod_{\omega_j=0}^{\ell-1}\ave[G_{\omega_j}^\ell]^{1/\ell}$.
Therefore, it is enough to show that for every $\omega_j\in\{0,\dots,\ell-1\}$ we have
\begin{equation} \label{e2.6}
\ave[G_{\omega_j}^\ell]\mik \prod_{\omega_{e'}\in\{0,\dots,\ell-1\}^{e'}}\|f_{(\omega_{e'},\omega_j)}\|_{\square_{\ell}(\bbx_e)}^\ell.
\end{equation}
Indeed, fix $\omega_j\in \{0,\dots,\ell-1\}$ and notice that, by \eqref{e2.5},
\begin{equation} \label{e2.7}
G_{\omega_j}^\ell\!(\bx_{e'}^{(0)},\dots,\bx_{e'}^{(\ell-1)})= \ave\Big[\prod_{\omega_{e'}\in\{0,\dots,\ell-1\}^{e'}}\prod_{\omega=0}^{\ell-1}
f_{(\omega_{e'}, \omega_j)}(\bx^{(\omega_{e'})}_{e'}, x^{(\omega)}_j)\Big]
\end{equation}
where the expectation is over all $x^{(0)}_j,\dots,x^{(\ell-1)}_j\in X_j$. By \eqref{e2.7} and Fubini's theorem, we see that
\[ \ave[G_{\omega_j}^\ell]=\ave\Big[\ave\big[\prod_{\omega_{e'}\in\{0,\dots,\ell-1\}^{e'}}\prod_{\omega=0}^{\ell-1}
f_{(\omega_{e'}, \omega_j)}(\bx_{e'}^{(\omega_{e'})}, x^{(\omega)}_j)\,\big|\, \bx_{e'}^{(0)},\dots,\bx_{e'}^{(\ell-1)}\in\bbx_{e'}\big]\Big] \]
where the outer expectation is over all $x^{(0)}_j,\dots,x^{(\ell-1)}_j\in X_j$. Thus, applying the induction hypothesis and H\"{o}lder's
inequality, we obtain that
\begin{eqnarray} \label{e2.8}
\ave[G_{\omega_j}^\ell] & \mik &
\ave\Big[\prod_{\omega_{e'}\in\{0,\dots,\ell-1\}^{e'}}\big\|\prod_{\omega=0}^{\ell-1}
f_{(\omega_{e'}, \omega_j)}(\,\cdot\,, x^{(\omega)}_j)\big\|_{\square_\ell(\bbx_{e'})}\Big] \\
&\mik & \prod_{\omega_{e'}\in\{0,\dots,\ell-1\}^{e'}} \ave\Big[\big\|\prod_{\omega=0}^{\ell-1}
f_{(\omega_{e'}, \omega_j)}(\,\cdot\,, x^{(\omega)}_j)\big\|^{\ell^{|e'|}}_{\square_\ell(\bbx_{e'})}\Big]^{1/\ell^{|e'|}}. \nonumber
\end{eqnarray}
By \eqref{e2.2} and \eqref{e2.8}, we conclude that \eqref{e2.6} is satisfied.
\medskip

\noindent (b) It is a consequence of the Gowers--Cauchy--Schwarz inequality. Specifically, for every $\omega\in\{0,\dots,\ell-1\}^e$
let $f_\omega=f$ if $\omega=0^e$ and $f_\omega=1$ otherwise. By \eqref{e2.3}, we see that $|\ave[f]|\mik\|f\|_{\square_\ell(\bbx_e)}$.
Next, let $\ell_1\mik\ell_2$ be even positive integers. As before, for every $\omega\in\{0,\dots,\ell_2-1\}^e$ let $f_\omega=f$ if
$\omega\in\{0,\dots,\ell_1-1\}^e$; otherwise, let $f_{\omega}=1$. Then we have
\begin{eqnarray*}
\|f\|^{\ell_1^{|e|}}_{\square_{\ell_1}(\bbx_e)}\!\!\!\! & = &
\ave\Big[\prod_{\omega\in\{0,\dots,\ell_1-1\}^e} \!\!\!f(\bx_e^{(\omega)})\,\Big|\, \bx_e^{(0)},\dots,\bx_e^{(\ell_1-1)}\in\bbx_e\Big] \\
& = & \ave\Big[\prod_{\omega\in\{0,\dots,\ell_2-1\}^e} \!\!\! f_\omega(\bx_e^{(\omega)})\,\Big|\, \bx_e^{(0)},\dots,\bx_e^{(\ell_2-1)}\in\bbx_e\Big]
\stackrel{\eqref{e2.3}}{\mik} \|f\|^{\ell_1^{|e|}}_{\square_{ \ell_2}(\bbx_e)}
\end{eqnarray*}
which implies that $\|f\|_{\square_{\ell_1}(\bbx_e)}\mik \|f\|_{\square_{\ell_2}(\bbx_e)}$.
\medskip

\noindent (c) Absolute homogeneity is straightforward. The triangle inequality
\[ \|f+g\|_{\square_\ell(\bbx_e)}\mik \|f\|_{\square_\ell(\bbx_e)}+\|g\|_{\square_\ell(\bbx_e)} \]
follows by raising both sides to the power $\ell^{|e|}$ and then applying \eqref{e2.3}. Finally, let $f\in L_1(\bbx_e,\bbs_e,\bmu_e)$
with $\|f\|_{\square_\ell(\bbx_e)}=0$ and observe that it suffices to show that $f=0$ $\bmu_e$-almost everywhere. First we note that using \eqref{e2.3}
and arguing precisely as in \cite[~Corollary B.3]{GT2} we have that $\ave[f\cdot \mathbf{1}_R]=0$ for every measurable rectangle $R$ of $\bbx_e$
(that is, every set $R$ of the form $\prod_{i\in e} A_i$ where $A_i\in \Sigma_i$ for every $i\in e$). We claim that this implies that
$\ave[f\cdot \mathbf{1}_A]=0$ for every $A\in\bbs_e$; this is enough to complete the proof. Indeed, fix $A\in\bbs_e$ and let $\ee>0$ be arbitrary.
Since $f$ is integrable, there exists $\delta>0$ such that $\ave[\,|f|\cdot \boldsymbol{1}_C]<\ee$ for every $C\in\bbs_e$ with
$\bmu_e(C)<\delta$. Moreover, by Caratheodory's extension theorem, there exists a finite family $R_1,\dots, R_m$ of pairwise disjoint measurable
rectangles of $\bbx_e$ such that, setting $B=\bigcup_{k=1}^m R_k$, we have $\bmu_e(A\,\triangle\, B)<\delta$ (see, e.g., \cite[Theorem 11.4]{Bi}).
Hence, $\ave[f\cdot\mathbf{1}_B]=0$ and so
\[ |\ave[f\cdot\mathbf{1}_A]|=|\ave[f\cdot\mathbf{1}_A]-\ave[f\cdot\mathbf{1}_B]|\mik \ave[\,|f|\cdot\mathbf{1}_{A\triangle B}]<\ee. \]
Since $\ee$ was arbitrary, we conclude that $\ave[f\cdot\mathbf{1}_A]=0$.
\medskip

\noindent (d) Set $I=\ave[f(x_i,x_j) u(x_i) v(x_j)\,|\,x_i \in X_i, x_j \in X_j]$ and let $\ell'$ denote the conjugate exponent of $\ell$.
Notice that $1<\ell'\mik p$. By H\"{o}lder's inequality, we have
\begin{eqnarray} \label{e2.9}
\ \ \ \ |I|\!\!\!\! & = & \big|\ave\big[ \ave[f(x_i,x_j)\, v(x_j)\,|\, x_j \in X_j]\, u(x_i) \,\big|\,x_i \in X_i\big]\big| \\
& \mik & \ave\big[ \ave[f(x_i,x_j)\, v(x_j)\,|\, x_j \in X_j]^{\ell}\,\big|\,x_i\in X_i\big]^{1/\ell}\cdot \|u\|_{L_{\ell'}}
\mik I_1^{1/\ell}\cdot \|u\|_{L_p} \nonumber
\end{eqnarray}
where $I_1=\ave\big[\prod_{\omega=0}^{\ell-1} f(x_i,x_j^{(\omega)})\, v(x_j^{(\omega)})\,\big| \, x_i \in X_i,x_j^{(0)},\dots,x_j^{(\ell-1)} \in X_j\big]$.
Moreover,
\begin{eqnarray} \label{e2.10}
\ \ \ \ \ \ \ I_1\!\!\!\!\!\!\! & = & \ave\Big[ \ave\big[\prod_{\omega=0}^{\ell-1}f(x_i,x_j^{(\omega)})\,\big|\, x_i \in X_i\big] \cdot
\prod_{\omega=0}^{\ell-1}v(x_j^{(\omega)}) \,\Big|\, x_j^{(0)},\dots,x_j^{(\ell-1)} \in X_j\Big] \\
& \mik & \ave\Big[ \ave\big[\prod_{\omega=0}^{\ell-1} f(x_i,x_j^{(\omega)})\,\big|\,x_i \in X_i\big]^{\ell}
\,\Big|\, x_j^{(0)},\dots,x_j^{(\ell-1)} \in X_j\Big]^{1/\ell}\cdot \|v\|_{L_{\ell'}}^{\ell} \nonumber \\
& \stackrel{\eqref{e2.2}}{=} & \|f\|_{\square_\ell(\bbx_e)}^\ell\cdot \|v\|_{L_{\ell'}}^{\ell}\mik
\|f\|_{\square_\ell(\bbx_e)}^\ell\cdot \|v\|_{L_p}^{\ell}. \nonumber
\end{eqnarray}
By \eqref{e2.9} and \eqref{e2.10}, the result follows.
\end{proof}

\subsection{The $(\ell,p)$-box norms}

We will need the following $L_p$ versions of the $\ell\text{-box}$ norms. We remark that closely related norms
appear\footnote{Precisely, in \cite{C}, for every finite abelian group $Z$, every integer $s\geqslant 2$
and every $f\colon Z\to \mathbb{R}$ the quantity $\big\|\, |f|^2 \big\|^{1/2}_{U^s(Z)}$ was considered. (Here,
$\|\cdot\|_{U^s(Z)}$ stands for the $s$-th Gowers uniformity norm for the group $Z$.) It is noted in \cite{C} that this quantity
is indeed a norm. The $(\ell,p)$-box norms defined above are the analogues, in the hypergraph setting, of these
norms.} in \cite{C}. Recall that by $\mathscr{H}~=~(n,\langle(X_i,\Sigma_i,\mu_i):i\in [n]\rangle,\calh)$
we denote a hypergraph system.
\begin{defn} \label{d2.2}
Let $e\subseteq [n]$ be nonempty and let $\ell\meg 2$ be an even integer. Also let $1\mik p<\infty$ and $f\in L_p(\bbx_e,\bbs_e,\bmu_e)$.
The \emph{$(\ell,p)$-box norm} of $f$ is defined by
\begin{equation} \label{e2.11}
\|f\|_{\square_{\ell,p}(\bbx_e)}\coloneqq \big\| |f|^p \big\|^{1/p}_{\square_{\ell}(\bbx_e)}.
\end{equation}
Moreover, for every $f\in L_\infty(\bbx_e,\bbs_e,\bmu_e)$ we set
\begin{equation} \label{e2.12}
\|f\|_{\square_{\ell,\infty}(\bbx_e)} \coloneqq \|f\|_{L_\infty}.
\end{equation}
\end{defn}
We have the following analogue of Proposition \ref{p2.1}.
\begin{prop} \label{p2.3}
Let $e\subseteq [n]$ be nonempty and let $\ell\meg 2$ be an even integer.
\begin{enumerate}
\item[(a)] Let  $1\mik p<\infty$. If $f_\omega\in L_p(\bbx_e,\bbs_e,\bmu_e)$ for every $\omega\in\{0,\dots,\ell-1\}^e$, then
\begin{equation} \label{e2.13}
\ave\Big[ \prod_{\omega\in\{0,\dots,\ell-1\}^e}\!\!\!\! |f_\omega|^p(\bx_e^{(\omega)})\,\Big|\, \bx_e^{(0)},\dots,\bx_e^{(\ell-1)}\in\bbx_e\Big]
\mik\!\!\! \prod_{\omega\in \{0,\dots,\ell-1\}^e}\!\!\!\!\!\! \|f_\omega\|^p_{\square_{\ell,p}(\bbx_e)}.
\end{equation}
\item[(b)] Let $1< p,q < \infty$ be conjugate exponents, that is, $1/p+1/q=1$. Then for every $f\in L_p(\bbx_e,\bbs_e,\bmu_e)$ and every
$g\in L_q(\bbx_e,\bbs_e,\bmu_e)$ we have
\begin{equation} \label{e2.14}
\|fg\|_{\square_{\ell}(\bbx_e)} \mik \|f\|_{\square_{\ell, p}(\bbx_e)}\!\cdot\, \|g\|_{\square_{\ell,q}(\bbx_e)}.
\end{equation}
\item[(c)] Assume that $|e|\meg 2$ and let $1\mik p<\infty$. Then $\|\cdot \|_{\square_{\ell, p}(\bbx_e)}$ is a norm on the vector subspace of\,
$L_p(\bbx_e,\bbs_e,\bmu_e)$ consisting of all $f\in L_p(\bbx_e,\bbs_e,\bmu_e)$ with $\|f\|_{\square_{\ell,p}(\bbx_e)}<\infty$.
Moreover, the following hold.
\begin{enumerate}
\item[(i)] For every $f\in L_p(\bbx_e,\bbs_e,\bmu_e)$ we have $\|f\|_{L_p}\mik \|f\|_{\square_{\ell,p}(\bbx_e)}$.
\item[(ii)] For every $1\mik p_1\mik p_2<\infty$ and every $f\in L_{p_2}(\bbx_e,\bbs_e,\bmu_e)$ we have $\|f\|_{\square_{\ell, p_1}(\bbx_e)}
\mik \|f\|_{\square_{\ell,p_2}(\bbx_e)}$.
\item[(iii)] For every $f\in L_\infty(\bbx_e,\bbs_e,\bmu_e)$ we have $\lim_{p\to\infty}\|f\|_{\square_{\ell, p}(\bbx_e)}=\|f\|_{L_\infty}$.
\end{enumerate}
\end{enumerate}
\end{prop}
\begin{proof}
Part (a) follows immediately by \eqref{e2.3}. For part (b) fix a pair $1<p,q<\infty$ of conjugate exponents, and let $f\in L_p(\bbx_e,\bbs_e,\bmu_e)$
and $g\in L_q(\bbx_e,\bbs_e,\bmu_e)$ be arbitrary. We define $F,G\colon\bbx_e^{\ell}\to \rr$ by
$F(\bx_e^{(0)},\dots,\bx_e^{(\ell-1)})=\prod_{\omega\in\{0,\dots,\ell-1\}^e}f(\bx_e^{(\omega)})$ and
$G(\bx_e^{(0)},\dots,\bx_e^{(\ell-1)})=\prod_{\omega\in\{0,\dots,\ell-1\}^e}g(\bx_e^{(\omega)})$. By H\"{o}lder's inequality, we have
\[ \|fg\|^{\ell^{|e|}}_{\square_{\ell}(\bbx_e)}\mik\ave[\,|F\cdot G|\,] \mik \ave[\,|F|^p]^{1/p} \cdot \ave[\,|G|^q]^{1/q}. \]
Noticing that $\ave[\,|F|^p]^{1/p}=\|f\|^{\ell^{|e|}}_{\square_{\ell,p}(\bbx_e)}$ and $\ave[\,|G|^q]^{1/q}=\|g\|^{\ell^{|e|}}_{\square_{\ell, q}(\bbx_e)}$
we conclude that \eqref{e2.14} is satisfied.

We proceed to show part (c). Arguing as in the proof of the classical Minkowski inequality we see that the $(\ell,p)$-box norm satisfies the
triangle inequality. Absolute homogeneity is clear and so, by Proposition \ref{p2.1}, we conclude that $\|\cdot\|_{\square_{\ell,p}(\bbx_e)}$
is indeed a norm. Next, observe that part (c.i) follows by \eqref{e2.13} applied for $f_\omega=f$ if $\omega=\{0\}^e$ and $f_\omega=1$ otherwise.
For part (c.ii) set $p=p_2/p_1$ and notice that
\[ \|f\|_{{\square}_{\ell,p_1}(\bbx_e)}^{p_1} =\big\| |f|^{p_1}\big\|_{\square_\ell(\bbx_e)}
\stackrel{\eqref{e2.14}}{\mik} \big\| |f|^{p_1}\big\|_{\square_{\ell,p}(\bbx_e)} =\|f\|^{p_1}_{\square_{\ell,p_2}(\bbx_e)}. \]
Finally, let $f\in L_\infty(\bbx_e,\bbs_e,\bmu_e)$. By part (c.i), we have
\[ \|f\|_{L_p}\mik \|f\|_{\square_{\ell,p}(\bbx_e)}\mik \|f\|_{L_\infty}.  \]
Since $\lim_{p\to\infty}\|f\|_{L_p}=\|f\|_{L_\infty}$, we obtain that
$\lim_{p\to\infty} \|f\|_{\square_{\ell, p}(\bbx_e)}=\|f\|_{L_\infty}$ and the proof is completed.
\end{proof}


\section{A generalized von Neumann theorem for $L_p$ graphons}

\numberwithin{equation}{section}

Let $(X,\Sigma,\mu)$ be a probability space and recall that a \emph{graphon}\footnote{We remark that in several
places in the literature, graphons are required to be $[0,1]$-valued, and the term \textit{kernel} is used for (not necessarily bounded)
integrable, symmetric random variables.} is an integrable random variable $W\colon X\times X\to\rr$ which is symmetric, that is, $W(x,y)=W(y,x)$
for every $x,y\in X$. If, in addition, $W$ belongs to $L_p$ for some $p>1$, then $W$ is said to be an \textit{$L_p$ graphon} (see \cite{BCCZ}).

Now let $n$ be a positive integer and let $\calg$ be a nonempty graph on $[n]$. Recall that the \emph{maximum degree} of $\calg$ is
the number $\Delta(\calg)\coloneqq \max\big\{ |\{e\in\calg:i\in e\}|:i\in [n]\big\}$. Given two $L_p$ graphons $W$ and $U$, a natural problem
(which is of importance in the context of graph limits---see, e.g., \cite{L}) is to estimate the quantity
\[ \Big|\ave\Big[ \prod_{\{i,j\}\in \calg} W(x_i,x_j)\, \Big| \, x_1,\dots,x_n\in X\Big] -
\ave\Big[ \prod_{\{i,j\}\in \calg} U(x_i,x_j)\, \Big| \, x_1,\dots,x_n\in X\Big] \Big|. \]
Note that this problem essentially boils down to that of analysing the boundedness of the multilinear operator
\[ \Lambda_{\calg}\big( (f_e)_{e\in\calg}\big) \coloneqq \ave\Big[ \prod_{e=\{i,j\}\in\calg} f_e(x_i,x_j)\, \Big| \, x_1,\dots,x_n\in X\Big] \]
where the functions $(f_e)_{e\in\calg}$ belong to $L_p$. Not surprisingly, the behavior of this operator depends heavily on the range
of $p$ one is working with. Undoubtedly, the simplest case is when $p=\infty$; indeed, using Fubini's theorem, it is not hard to see
that for bounded functions the operator $\Lambda_{\calg}$ is controlled by the cut norm\footnote{We recall the definition of the cut
norm in \eqref{e4.2}. We note, however, that we will not use the cut norm in this section.}. The next critical range for the behavior
of $\Lambda_{\calg}$ is when $\Delta(\calg)\mik p<\infty$. In this case, H\"{o}lder's inequality yields that $\Lambda_{\calg}$ is
bounded in $L_p$. This information was used in \cite[~Theorem 2.20]{BCCZ} to show that $\Lambda_{\calg}$ is also controlled by the
cut norm when~$p>\Delta(\calg)$.

Unfortunately, in the regime $1<p<\Delta(\calg)$ the operator $\Lambda_{\calg}$ is not bounded but merely densely defined in $L_p$.
Nevertheless, experience from arithmetic combinatorics (see, e.g., \cite{Go1,TV}) and harmonic analysis (see, e.g., \cite{K})
indicates that one can still obtain nontrivial information provided that one replaces the $L_p$-norm with a suitable box norm.
It turns out that this intuition is correct as is shown in the following theorem.
\begin{thm}[Generalized von Neumann theorem for $L_p$ graphons] \label{t3.1}
Let $\Delta$ be a positive integer, $C\meg 1$ and $1<p\mik\infty$. If\, $p=\infty$ or $\Delta=1$, then we set $\ell=2$; otherwise, let
\begin{equation} \label{e3.1}
\ell=\min\big\{2n:n\in\nn \text{ and }  2n \meg p^{(\Delta-1)^{-1}}(p^{(\Delta-1)^{-1}}-1)^{-1}\big\}.
\end{equation}
Also let $\mathscr{G}=(n,\langle (X_i,\Sigma_i,\mu_i):i\in [n]\rangle,\calg)$ be a $2$-uniform hypergraph system with $\Delta(\calg)=\Delta$.
For every $e \in \calg$ let $f_e\in L_p(\bbx,\calb_e,\bmu)$ such that
\begin{equation} \label{e3.2}
\|\mathbf{f}_e\|_{\square_{\ell, p}(\bbx_e)}\mik 1
\end{equation} where $\mathbf{f}_e$
is as in \eqref{e1.5}. Assume that for every $($possibly empty$)$ $\calg'\subseteq \calg$ we have
\begin{equation}\label{e3.3}
\big\| \prod_{e \in \calg'} f_e \big\|_{L_p} \mik C.
\end{equation}
$($Here, we follow the convention that the product of an empty family of functions is equal to the constant function 1.$)$ Then we have
\begin{equation}\label{e3.4}
\Big|\ave\Big[\prod_{e \in \calg}f_e\Big] \Big|\mik C\cdot \min_{e\in \calg} \|\mathbf{f}_e\|_{\square_\ell(\bbx_e)}.
\end{equation}
\end{thm}
Observe that \eqref{e3.2} is an integrability condition; as we have already noted, this condition is necessary if $p<\Delta$.
On the other hand, condition \eqref{e3.3} is the analogue of the ``linear forms condition" appearing in several versions
of the generalized von Neumann theorem (see, e.g., \cite[Proposition 5.3]{GT1} and \cite[Theorem 3.8]{Tao2}).
\begin{proof}[Proof of Theorem \emph{\ref{t3.1}}]
Let $e\in\calg$ be arbitrary, and set
\[ I\coloneqq \ave\Big[ f_e\!\!\prod_{e' \in \calg\setminus\{e\}}\!\! f_{e'}\Big] \]
Clearly, it suffices to show that $|I| \mik  C\cdot \|\mathbf{f}_{e}\|_{\square_{\ell}(\bbx_e)}$.

To this end, we first observe that if $\Delta=1$, then the result is straightforward. Indeed, in this case we have $\ell=2$ and the edges
of $\calg$ are pairwise disjoint. Hence, by part~(b) of Proposition~\ref{p2.1} and part (c.ii) of Proposition \ref{p2.3}, we see that
\[ |I|= |\ave[\mathbf{f}_e]|\,\cdot\!\!\!\! \prod_{e'\in\calg\setminus\{e\}}\!\! |\ave[\mathbf{f}_{e'}]| \mik
\|\mathbf{f}_e\|_{\square_2(\bbx_e)}\, \cdot\!\!\!\! \prod_{e' \in \calg\setminus\{e\}}\!\! \|\mathbf{f}_{e'}\|_{\square_{2,p}(\bbx_{e'})}
\stackrel{\eqref{e3.2}}{\mik} C \cdot \|\mathbf{f}_e\|_{\square_2(\bbx_e)}. \]
Therefore, in what follows we will assume that $\Delta\meg 2$. To simplify the exposition we will also assume that $p\neq \infty$.
(The proof for the case $p=\infty$ is similar.) Write $e=\{i,j\}$, and set $\calg(i)=\{e'\in\calg\setminus\{e\}: i\in e'\}$ and
$\calg^*(i)=\{e'\in\calg\setminus\{e\}: i\notin e'\}$; notice that $\calg\setminus\{e\}=\calg(i)\cup \calg^*(i)$. Let $\ell'$
be the conjugate exponent of $\ell$  and observe that, by~\eqref{e3.1}, we have $\ell\meg q'$ where $q'$ is the conjugate
exponent of $p^{(\Delta-1)^{-1}}$. Hence,
\begin{equation} \label{e3.5}
1<\ell'\mik p^{(\Delta-1)^{-1}}\mik p.
\end{equation}
We set
\begin{equation} \label{e3.6}
I_{e,\calg(i)}= \ave\Big[\prod_{\omega=0}^{\ell-1}\mathbf{f}_{e}(x_i^{(\omega)},x_j)\!\! \prod_{e'\in \calg(i)}
\mathbf{f}_{e'}(x_i^{(\omega)}, x_{e'\setminus\{i\}})\Big]
\end{equation}
and
\begin{equation}\label{e3.7}
I_{\calg(i)}= \ave\Big[\prod_{e'\in\calg(i)}\prod_{\omega=0}^{\ell-1} |\mathbf{f}_{e'}|^{\ell'}(x_i^{(\omega)},x_{e'\setminus\{i\}})\Big]
\end{equation}
where both expectations are over all $x_i^{(0)},\dots, x_i^{(\ell-1)}\in X_i$ and $\bx_{[n]\setminus \{i\}}\in \bbx_{[n]\setminus \{i\}}$.
\begin{claim} \label{c3.2}
We have  $|I|\mik C\cdot I_{e,\calg(i)}^{1/\ell} $.
\end{claim}
\begin{proof}[Proof of Claim \emph{\ref{c3.2}}]
Since $i\notin e'$ for every $e'\in \calg^*(i)$, we have
\[ I=\ave\Big[\ave\big[\mathbf{f}_e(x_i, x_j)\!\!\prod_{e'\in\calg(i)}\mathbf{f}_{e'}(x_i, x_{e'\setminus\{i\}})\,\big|\, x_i \in X_i\big]
\,\cdot\!\!\! \prod_{e'\in \calg^*(i)}\!\!\mathbf{f}_{e'}(\bx_{e'})\Big]. \]
By H\"older's inequality, \eqref{e3.3}, \eqref{e3.5} and \eqref{e3.6}, we obtain that
\begin{eqnarray*}
|I| & \mik & \ave\Big[\ave\big[\mathbf{f}_e(x_i, x_j)\!\!\prod_{e'\in \calg(i)}\mathbf{f}_{e'}(x_i, x_{e'\setminus\{i\}})\,\big|\,
x_i\in X_i\big]^\ell\Big]^{1/\ell}\cdot \big\|\!\prod_{e'\in \calg^*(i)} f_{e'}\big\|_{L_{\ell'}} \\
& \mik & I_{e,\calg(i)}^{1/\ell} \cdot \big\|\! \prod_{e'\in \calg^*(i)} f_{e'}\big\|_{L_{p}} \mik
C\cdot I_{e,\calg(i)}^{1/\ell}
\end{eqnarray*}
as desired.
\end{proof}
We proceed with the following claim.
\begin{claim} \label{c3.3}
We have $I_{e,\calg(i)}\mik\|\mathbf{f}_{e}\|_{\square_{\ell}(\bbx_e)}^{\ell}\cdot I_{\calg(i)}^{1/\ell'}$.
\end{claim}
\begin{proof}[Proof of Claim \emph{\ref{c3.3}}]
Note that $j\notin e'$ for every $e'\in \calg(i)$, and so
\[ I_{e,\calg(i)} = \ave\Big[\ave\big[\prod_{\omega=0}^{\ell-1}\mathbf{f}_{e}(x_i^{(\omega)},x_j)\,\big|\, x_j \in X_j\big]
\,\cdot\!\! \prod_{e'\in \calg(i)} \prod_{\omega=0}^{\ell-1} \mathbf{f}_{e'}(x_i^{(\omega)}, x_{e'\setminus\{i\}})\Big]. \]
Using this observation the claim follows by H\"older's inequality and arguing precisely as in the proof of Claim \ref{c3.2}.
\end{proof}
The following claim is the last step of the proof.
\begin{claim} \label{c3.4}
We have $I_{\calg(i)}\mik 1$.
\end{claim}
\begin{proof}[Proof of Claim \emph{\ref{c3.4}}]
We may assume, of course, that $\calg(i)$ is nonempty. We set $m=|\calg(i)|$ and we observe that $1\mik m\mik \Delta-1$.
Therefore, by \eqref{e3.5}, we see that
\begin{equation} \label{e3.8}
1<(\ell')^r\mik (\ell')^{\Delta-1}\mik p
\end{equation}
for every $r\in [m]$. Write $\calg(i)=\{e'_1,\dots,e'_m\}$ and for every $r\in [m]$ let $j_r\in [n]$ such that
$e'_r=\{i,j_r\}$. For every $d\in [m]$ set
\begin{equation} \label{e3.9}
Q_d=\ave\Big[\prod_{r=d}^m\, \prod_{\omega=0}^{\ell-1} |\mathbf{f}_{e'_r}|^{(\ell')^d}(x_i^{(\omega)},x_{j_r}) \Big]
\end{equation}
and note that
\begin{equation} \label{e3.10}
Q_1=I_{\calg(i)} \ \text{ and } \ Q_m=\ave\Big[\prod_{\omega=0}^{\ell-1}|\mathbf{f}_{e'_m}|^{(\ell')^m}(x_i^{(\omega)},x_{j_m})\Big].
\end{equation}
(Here, the expectation is over all $x_i^{(0)},\dots, x_i^{(\ell-1)}\in X_i$ and $\bx_{[n]\setminus \{i\}}\in \bbx_{[n]\setminus \{i\}}$.)
Now observe that it is enough to show that for every $d\in [m-1]$ we have
\begin{equation} \label{e3.11}
Q_d \mik Q_{d+1}^{1/\ell'}.
\end{equation}
Indeed, by \eqref{e3.11}, we see that $Q_1\mik Q_m^{1/(\ell')^{m-1}}$. Hence, by \eqref{e3.10}, the monotonicity of the $L_p$ norms
and part (a) of Proposition \ref{p2.3}, we obtain that
\begin{eqnarray} \label{e3.12}
I_{\calg(i)} & \mik & \ave\Big[\prod_{\omega=0}^{\ell-1}|\mathbf{f}_{e'_m}|^{(\ell')^m}(x_i^{(\omega)},x_{j_m})\Big]^{\ell'/(\ell')^m} \\
& \stackrel{ \eqref{e3.8}}{\mik} & \ave\Big[\prod_{\omega=0}^{\ell-1}|\mathbf{f}_{e'_m}|^{p}(x_i^{(\omega)},x_{j_m})\Big]^{\ell'/p} \mik
\|\mathbf{f}_{e'_m}\|_{\square_{\ell,p}(\bbx_{e'_m})}^{\ell\ell'} \stackrel{\eqref{e3.2}}{\mik} 1. \nonumber
\end{eqnarray}
It remains to show \eqref{e3.11}. Fix $d\in [m-1]$ and notice that $j_d\notin e'_r$ for every $r\in \{d+1,\dots,m\}$. Thus,
\[ Q_d= \ave\Big[ \ave\big[\prod_{\omega=0}^{\ell-1}|\mathbf{f}_{e'_d}|^{(\ell')^d}(x_i^{(\omega)},x_{j_d})\,\big| \, x_{j_d} \in X_{j_d}\big]
\,\cdot\!\!\prod_{r=d+1}^m \prod_{\omega=0}^{\ell-1} |\mathbf{f}_{e'_r}|^{(\ell')^d}(x_i^{(\omega)},x_{j_r}) \Big]. \]
By H\"older's inequality and arguing as in the proof of \eqref{e3.12}, we see that
\[ Q_d \mik \ave\Big[\prod_{\omega\in\{0,\dots,\ell-1\}^{e'_{\!d}}}\!\!\! |\mathbf{f}_{e'_d}|^{(\ell')^d}(\bx_{e'_d}^{(\omega)})\Big]^{1/\ell}
\cdot Q_{d+1}^{1/\ell'} \mik \|\mathbf{f}_{e'_d}\|_{\square_{\ell,p}(\bbx_{e'_{\!d}})}^{\ell(\ell')^d} \cdot Q_{d+1}^{1/\ell'} \]
as desired.
\end{proof}
By Claims \ref{c3.2}, \ref{c3.3} and \ref{c3.4}, we conclude that $|I| \mik  C\cdot \|\mathbf{f}_{e}\|_{\square_{\ell}(\bbx_e)}$,
and so the entire proof of Theorem \ref{t3.1} is completed.
\end{proof}
We close this section with the following counting lemma for $L_p$~graphons. It follows readily by Theorem \ref{t3.1} and a telescopic argument.
\begin{cor} \label{c3.5}
Let $\Delta, C, p, \ell$ and $\mathscr{G}$ be as in Theorem \emph{\ref{t3.1}}. For every $e \in \calg$ let $f_e, g_e \in L_p(\bbx,\calb_e,\bmu)$ such that
$\|\mathbf{f}_e\|_{\square_{\ell, p}(\bbx_e)}\mik 1$ and $\|\mathbf{g}_e\|_{\square_{\ell, p}(\bbx_e)}\mik 1$ where $\mathbf{f}_e$ and $\mathbf{g}_e$ are
as in \eqref{e1.5} for $f_e$ and $g_e$ respectively. Assume that for every $\calg_1,\calg_2 \subseteq \calg$ with $\calg_1 \cap \calg_2=\emptyset$ we have
$\| \prod_{e \in \calg_1} f_e \prod_{e \in \calg_2}g_e\|_{L_p} \mik C$. Then we have
\begin{equation}\label{e3.13}
\Big|\ave\Big[\prod_{e \in \calg}f_e\Big]-\ave\Big[\prod_{e \in \calg}g_e\Big]\Big|\mik
C\cdot \sum_{e\in\calg}\|\mathbf{f}_e - \mathbf{g}_e\|_{\square_\ell(\bbx_e)}.
\end{equation}
\end{cor}


\section{Pseudorandom families}

\numberwithin{equation}{section}

\subsection{ \ }

We begin by introducing some pieces of notation. Let $n,r$ be two positive integers with $n\meg r\meg 2$ and let
$\mathscr{H}=(n,\langle(X_i,\Sigma_i,\mu_i):i\in [n]\rangle,\calh)$ be an $r$-uniform hypergraph system.
Given $e\in\calh$ let $\partial e=\{e'\subseteq e: |e'|=|e|-1\}$ and set
\begin{equation} \label{e4.1}
\cals_{\partial e}\coloneqq \Big\{ \bigcap_{e'\in\partial e} A_{e'}: A_{e'}\in\calb_{e'} \text{ for every } e'\in\partial e\Big\}
\subseteq \calb_e.
\end{equation}
Also recall that for every $f\in L_1(\bbx,\calb_e,\bmu)$ the \emph{cut norm} of $f$ is defined by
\begin{equation} \label{e4.2}
\|f\|_{\cals_{\partial e}} =\sup\Big\{ \big| \int_A f\, d\bmu \big|: A\in\cals_{\partial e}\Big\}.
\end{equation}
The cut norm is a standard tool in extremal combinatorics (see, e.g., \cite{FK,L,Tao1}). It is weaker than the box norm,
but for bounded functions these two norms are essentially equivalent (see \cite[Theorem 4.1]{Go2}).

The following class of sparse weighted uniform hypergraphs was introduced in \cite[Definition 6.1]{DKK}.
\begin{defn} \label{d4.1}
Let $n,r\in\nn$ with $n\meg r\meg 2$, and let $C\meg 1$ and $0<\eta<1$. Also let $1<p\mik\infty$ and let $q$ denote the conjugate
exponent of $p$. Finally, let $\mathscr{H}=(n,\langle(X_i,\Sigma_i,\mu_i):i\in [n]\rangle,\calh)$ be an $r$-uniform hypergraph
system. For every $e\in\calh$ let $\nu_e\in L_1(\bbx,\calb_e,\bmu)$ be a nonnegative random variable. We say that the family
$\langle \nu_e:e\in\calh\rangle$ is \emph{$(C,\eta,p)$-pseudorandom} if the following hold.
\begin{enumerate}
\item[(C1)] For every nonempty $\calg\subseteq\calh$ we have $\ave\big[\prod_{e\in\calg} \nu_e\big]\meg 1-\eta$.
\item[(C2)] For every $e\in\calh$ there exists $\psi_e\in L_p(\bbx,\calb_e,\bmu)$ with $\|\psi_e\|_{L_p} \mik C$
and satisfying the following properties.
\begin{enumerate}
\item[(a)] We have $\|\nu_e-\psi_e\|_{\cals_{\partial e}}\mik \eta$.
\item[(b)] For every $e'\in\calh\setminus\{e\}$ and every $\omega\in\{0,1\}$ let $g_{e'}^{(\omega)}\in L_1(\bbx,\calb_{e'},\bmu)$
such that either $0 \mik g_{e'}^{(\omega)} \mik \nu_{e'}$ or\, $0 \mik g_{e'}^{(\omega)}\mik 1$. Let $\boldsymbol{\nu}_e$ and
$\boldsymbol{\psi}_e$ be as in \eqref{e1.5} for $\nu_e$ and $\psi_e$ respectively. Then we have
\[ \Big|\ave\Big[ (\boldsymbol{\nu}_e-\boldsymbol{\psi}_e)(\bx_e)\!\! \prod_{\omega\in\{0,1\}}
\ave\big[\!\!\!\!\!\!\prod_{e'\in\calh\setminus\{e\}}\!\!\!\!\!\! g_{e'}^{(\omega)}(\bx_e,\bx_{[n]\setminus e})\,\big|\,
\bx_{[n]\setminus e}\in \boldsymbol{X}_{[n]\setminus e}\big] \Big|\, \bx_e\in\boldsymbol{X}_e\Big]\Big|\mik \eta. \]
\end{enumerate}
\item[(C3)] Let $e\in\calh$ and let $\calg\subseteq\calh\setminus \{e\}$ be nonempty, and define
$\boldsymbol{\nu}_{e,\calg}\colon\bbx_e \to \rr$ by $\boldsymbol{\nu}_{e,\calg}(\bx_e)=\ave\big[ \prod_{e'\in\calg}
\nu_{e'}(\bx_e,\bx_{[n]\setminus e}) \, \big|\, \bx_{[n]\setminus e}\in \boldsymbol{X}_{[n]\setminus e}\big]$. Then, setting
\begin{equation} \label{e4.3}
\ell \coloneqq \min\Big\{2n :n\in\nn \text{ and } 2n \meg 2q + \Big(1 -\frac{1}{C }\Big) + \frac{1}{p}\Big\}
\end{equation}
$($where $1/p=0$ if $p=\infty$$)$, we have\, $\ave[\boldsymbol{\nu}_{e,\calg}^{\ell}]\mik C+\eta$.
\end{enumerate}
\end{defn}
We note that closely related definitions were introduced in \cite{CFZ1,Tao2} and we refer the reader to \cite[Section 6]{DKK} for a detailed
discussion on conditions (C1)--(C3) and their relation with the notions of pseudorandomness appearing in \cite{CFZ1,Tao2}. As we have mentioned
in the introduction, the most important property of pseudorandom families is that they satisfy relative versions of the counting and removal lemmas;
see, in particular, \cite[Theorems 2.2~and 7.1]{DKK}.

\subsection{Motivation}

In the second part of this paper our goal is to give examples of pseudorandom families.
We have already pointed out that these examples can be seen as deviations (in an $L_p$-sense) of weighted hypergraphs which satisfy
the linear forms condition. This fact is not accidental. Indeed, by \cite[Theorem 2.1]{DKK}, under quite general hypotheses one can
decompose a nonnegative random variable $\nu_e$ as $s_e+ u_e$ where $s_e$ belongs to $L_p$ and $u_e$ has negligible cut norm. Unfortunately,
this information is not strong enough to yield that the weighted hypergraph $\langle \nu_e:e\in\calh\rangle$ satisfies relative versions
of the counting and removal lemmas. However, as we shall see, this problem can be bypassed by imposing slightly stronger integrability
conditions on each $s_e$ and assuming that the random variables $\langle u_e:e\in\calh\rangle$ are very mildly correlated.

\subsection{The first main result}

The following theorem is our first result in this section. Its proof is given in Section 5.
\begin{thm} \label{t4.2}
Let $n\in\nn$ with $n\meg 3$, $C\meg 1$ and $1<p\mik\infty$, and let $\ell$ be as~in~\eqref{e4.3}. Also let $0<\eta\mik (4C)^{-n\ell^n}$
and let $\mathscr{H}=(n,\langle(X_i,\Sigma_i,\mu_i):i\in [n]\rangle,\calh)$ be a hypergraph system with $\calh=K^{(n-1)}_n={[n]\choose n-1}$.
$($In particular, we have that $\mathscr{H}$ is $(n-1)$-uniform.$)$ For every $e \in \calh$ let $\lambda_e\in L_1(\bbx,\calb_e,\bmu)$ and
$\varphi_e\in L_p(\bbx,\calb_e,\bmu)$ be nonnegative random variables, and let $\boldsymbol{\lambda}_e$ and $\boldsymbol{\varphi}_e$
be as in \eqref{e1.5} for $\lambda_e$ and $\varphi_e$ respectively. Assume that the following conditions are satisfied.
\begin{enumerate}
\item[(I)] We have
\begin{equation} \label{e4.4}
1-\eta \mik \ave\Big[ \prod_{e \in \calh} \ \prod_{\omega \in \{0,\ldots,\ell -1 \}^e }\!\! \boldsymbol{\lambda}_e ^{n_{e,\omega}}
(\bx_e^{(\omega)})\, \Big|\, \bx^{(0)},\dots,\bx^{(\ell-1)}\in \boldsymbol{X}\Big]\mik 1 + \eta
\end{equation}
for any choice of $n_{e,\omega}\in \{0,1\}$.
\item[(II)] For every $e\in\calh$ we have $\|\boldsymbol{\varphi}_e\|_{\square_{\ell,p}(\bbx_e)}\mik C$.
\end{enumerate}
Then the family $\langle \lambda_e+\varphi_e:e\in \calh\rangle$ is $(C',\eta',p)$-pseudorandom where $C'=(4C)^{n\ell}$ and
$\eta'=(4C)^{n\ell}\,\eta^{1/\ell^{n-1}}$.
\end{thm}
We remark that condition (I) in Theorem \ref{t4.2} is a modification of the \textit{linear forms condition} introduced in \cite[Definition 2.8]{Tao2};
it expresses the fact that the weighted hypergraph $\langle\lambda_e:e\in\calh\rangle$ contains roughly the expected number of copies of the
$\ell$-blow-up of $\calh$ and its sub-hypergraphs. On the other hand, note that condition (II) is an integrability condition; in particular,
using H\"{o}lder's inequality, it is easy to see that $\|\boldsymbol{\varphi}_e\|_{\square_{\ell,p}(\bbx_e)}\mik C$ provided that
$\|\boldsymbol{\varphi}_e\|_{L_q}\mik C$ for some $q$ sufficiently large. Thus we see that the family $\langle \nu_e+\varphi_e:e\in\calh\rangle$
is a perturbation of a system of measures which appears in \cite{Tao2}, the main point being that only integrability conditions are imposed
on each ``noise" $\varphi_e$.

We proceed to give some concrete examples of weighted graphs and hypergraphs which are obtained using Theorem \ref{t4.2}.
They are the simplest type of examples for which the results obtained in \cite{DKK} can be applied, yet they are out of the scope
of the counting lemmas developed by Tao \cite{Tao2}, and by Conlon, Fox and Zhao \cite{CFZ1}.

\subsubsection{Example: weighted graphs}

Let $V_1,V_2,V_3$ be three pairwise disjoint nonempty sets; we view $V_1,V_2$ and $V_3$ as discrete probability spaces equipped with their
uniform probability measures. Also let $\calh$ denote the graph $K_3=\big\{\{1,2\},\{2,3\},\{1,3\}\big\}$ (that is, $\calh$ is the complete
graph on three vertices). For every $e=\{i<j\}\in\calh$ let $\boldsymbol{\varphi}_e\colon V_i\times V_j\to \rr^+$ be \textit{any} function
satisfying\footnote{We do not know whether the estimate \eqref{e4.5} is optimal; in fact, it is likely that the exponent $64$
can be improved. We point out, however, that an integrability condition like \eqref{e4.5} is necessary in order to have a sparse
version of the counting lemma.}
\begin{equation} \label{e4.5}
\|\boldsymbol{\varphi}_e\|_{L_{64}} = \ave \big[\boldsymbol{\varphi}_e(x,y)^{64} \, \big|\, (x,y)\in V_i \times V_j\big]^{1/64}\mik 1
\end{equation}
and define
\begin{equation} \label{e4.6}
\boldsymbol{\lambda}_e=1 \ \text{ and } \ \boldsymbol{\nu}_e=1+\boldsymbol{\varphi}_e.
\end{equation}

\begin{figure}[htb] \label{figure1}
\centering \includegraphics[width=.55\textwidth]{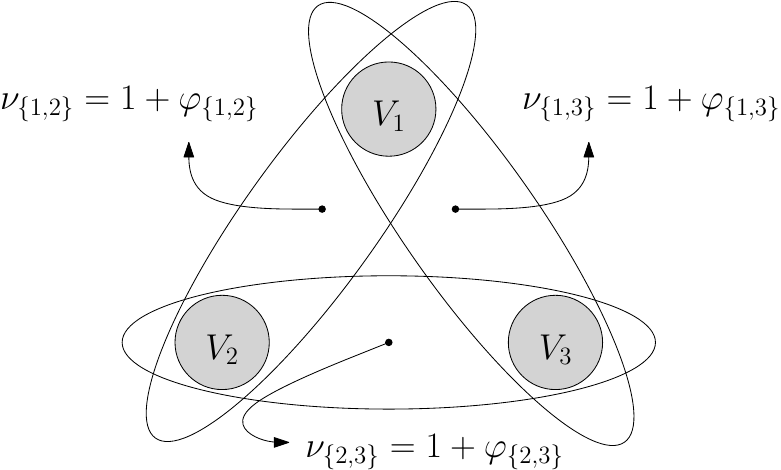}
\end{figure}

We will apply Theorem \ref{t4.2} for $n=3$, $C=1$ and $p=4$ in order to show that the weighted graph
$\langle \nu_e:e\in\calh\rangle$ is a $(4^{12},\eta,4)$-pseudorandom family for every $\eta>0$.
(Here, $\nu_e$ is as in \eqref{e1.5} for $\boldsymbol{\nu}_e$.) To this end notice first that,
by \eqref{e4.3}, we have $\ell=4$. On the other hand, it is clear that \eqref{e4.4}---that is,
condition (I) in Theorem \ref{t4.2}---is satisfied for every $\eta>0$. Finally, for condition (II)
observe that, by H\"{o}lder's inequality\footnote{Note that here, as in the proof of part (a) of Proposition \ref{p2.1},
we use the generalized form of H\"{o}lder's inequality.}, for every $e=\{i<j\}\in\calh$ we have
\[ \|\boldsymbol{\varphi}_e\|_{\square_{4,4}(V_i\times V_j)} \mik
\|\boldsymbol{\varphi}_e^4\|_{L_{16}} = \|\boldsymbol{\varphi}_e\|_{L_{64}}^4 \stackrel{\eqref{e4.5}}{\mik} 1. \]
Thus, condition (II) is also satisfied, which yields that $\langle \nu_e:e\in\calh\rangle$ is indeed
a $(4^{12},\eta,4)$-pseudorandom family for every $\eta>0$.

Our next goal is to show that the weighted graphs defined above cannot be realized as dense subgraphs of weighted graphs
which satisfy the aforementioned ``linear forms condition"\!, a pseudorandomness condition which forms the basis of the sparse
counting lemmas developed in \cite{CFZ1, Tao2}. The framework in \cite{CFZ1,Tao2} is asymptotic; consequently,
for every integer $N\meg 1$ we select, recursively, a positive integer $m_N$, three pairwise disjoint nonempty sets
$V_1^N,V_2^N,V_3^N$ and families $\langle A^N_{1,l}: l\in [N]\rangle$, $\langle A^N_{2,l}: l\in [N]\rangle$,
$\langle A^N_{3,l}: l\in [N]\rangle$ of nonempty sets such that the following hold.
\begin{enumerate}
\item[($\mathcal{P}$1)] We have $m_{N+1}^{64} \meg 2^{N+2}\, m_N$; moreover, $m_1=2$.
\item[($\mathcal{P}$2)] For every $i\in\{1,2,3\}$ and every $l\in [N]$ we have $A^N_{i,l}\subseteq V^N_i$.
\item[($\mathcal{P}$3)] For every $e=\{i<j\}\in \calh$ and every $l\in [N]$ we have
\[ |A^N_{i,l} \times A^N_{j,l}| = \Big( \frac{1}{m^{64^2}_l} \Big)\, |V^N_i \times V^N_j|. \]
In particular, we have $\|\mathbf{1}_{A^N_{i,l}\times A^N_{j,l}}\|_{L_{64}}=1/m_l^{64}$.
\end{enumerate}
Given the above data, for every integer $N\meg 1$ and every $e=\{i<j\}\in \calh$ we define
$\boldsymbol{\varphi}_e^N,\boldsymbol{\lambda}^N_e,\boldsymbol{\nu}_e^N\colon V^N_i\times V^N_j\to\rr^+$ by setting
\begin{eqnarray}
\label{e4.7}
& &  \boldsymbol{\varphi}_e^N=\sum_{l=1}^{N} \frac{1}{2^l}\,
\frac{ \mathbf{1}_{A^N_{i,l}\times A^N_{j,l}} }{ \|\mathbf{1}_{A^N_{i,l}\times A^N_{j,l}}\|_{L_{64}}} =
\sum_{l=1}^{N} \frac{m_l^{64}}{2^l}\, \mathbf{1}_{A^N_{i,l}\times A^N_{j,l}},  \\
& & \ \ \ \ \ \ \ \ \ \ \ \ \ \ \boldsymbol{\lambda}^N_e=1 \ \text{ and } \ \boldsymbol{\nu}^N_e=1+\boldsymbol{\varphi}^N_e.  \nonumber
\end{eqnarray}
Since $\|\boldsymbol{\varphi}_e^N\|_{L_{64}}\mik 1$, by the previous discussion, we see that the weighted graph
 $\langle \nu^N_e:e\in\calh\rangle$ is a $(4^{12},\eta,4)$-pseudorandom family for every $\eta>0$ and every $N\meg 1$.

Now assume, towards a contradiction, that there exist a constant $M>0$ and a sequence $\langle \kappa_e^N:e\in\calh\rangle$
of weighted graphs\footnote{Specifically, for every $e=\{i<j\}\in \calh$ and every $N\meg 1$ we have
$\boldsymbol{\kappa}_e^N\colon V^N_i\times V^N_j\to\rr^+$, and $\kappa^N_e$ is as in \eqref{e1.5} for $\boldsymbol{\kappa}_e^N$.}
which satisfy the ``linear forms condition" (see \cite[Definition 2.8]{CFZ1} or \cite[Definition 2.8]{Tao2}) such that
$\nu^N_{e_0} \mik M \cdot \kappa_{e_0}^N$ for some $e_0=\{i<j\}\in\calh$ and infinitely many $N$.
Setting $l_0\coloneqq \min\{l\meg 1: M\mik m_l\}$, we thus have\footnote{Note that here we do not use the fact that
$\boldsymbol{\lambda}^N_e=1$ for every $e\in\calh$ and every integer $N\meg 1$. Actually, the same argument can be applied
if $\langle\boldsymbol{\lambda}^N_e:e\in\calh\rangle$ is \textit{any} sequence of weighted graphs which satisfy condition (I)
in Theorem \ref{t4.2} for $\ell=4$ and $\eta>0$ sufficiently small.}
\begin{equation} \label{e4.8}
0\mik \boldsymbol{\varphi}^N_{e_0} \mik m_{l_0} \cdot \boldsymbol{\kappa}_{e_0}^N \ \text{ for infinitely many $N$.}
\end{equation}
Next, set
\[ \sigma=\frac{1}{2}\, \frac{(m^{64}_{l_0+1}/ 2^{l_0+1}) - m_{l_0}}{m^{64^2}_{l_0+1}} \]
and notice that, by ($\mathcal{P}$1) above, we have $\sigma>0$. By \eqref{e4.8} and \cite[Lemma 2.15 and Theorem 2.16]{CFZ1},
for every integer $N\meg 1$ there exists $\boldsymbol{h}_{e_0}^N\colon V_i^N\times V_j^N\to \rr$
with $0\mik \boldsymbol{h}_{e_0}^N \mik m_{l_0}$ and such that
\begin{equation} \label{e4.9}
\!\!\sup\big\{ \big|\ave\big[(\boldsymbol{\varphi}_{e_0}^N-\boldsymbol{h}^N_{e_0})\,  \boldsymbol{1}_{A\times B}\big]\big| :
A\subseteq V_i^N, B\subseteq V_j^N \big\} \mik \sigma \ \text{ for infinitely many $N$.}
\end{equation}
On the other hand, by ($\mathcal{P}2$), ($\mathcal{P}3$), \eqref{e4.7} and the fact $0\mik \boldsymbol{h}_{e_0}^N \mik m_{l_0}$,
for every $N\meg l_0+1$ we have
\[ \frac{(m^{64}_{l_0+1}/ 2^{l_0+1}) - m_{l_0}}{ m^{64^2}_{l_0+1}} \mik
\big|\ave\big[ (\boldsymbol{\varphi}_{e_0}^N-\boldsymbol{h}^N_{e_0}) \, \boldsymbol{1}_{A^N_{i,l_0+1}\times A^N_{j,l_0+1}}\big]\big|  \]
which clearly leads to a contradiction by \eqref{e4.9} and the choice of $\sigma$.

\subsubsection{Example: weighted hypergraphs}

It is the hypergraph analogue of the previous example. Specifically, let $n\meg 3$ be an integer, and let
$V_1,\dots,V_n$ be pairwise disjoint nonempty sets which we view as discrete probability spaces equipped
with their uniform probability measures. Also let $\calh$ denote the graph $K^{(n-1)}_n$ (thus, $\calh$ is the complete
$(n-1)$-uniform hypergraph on $n$ vertices). For every $e\in\calh$ set $\mathbf{V}_e=\prod_{i\in e} V_i$, let
$\boldsymbol{\varphi}_e\colon \mathbf{V}_e\to \rr^+$ be a function satisfying
\begin{equation} \label{e4.10}
\|\boldsymbol{\varphi}_e\|_{L_{4^n}} = \ave \big[\boldsymbol{\varphi}_e(\bx_e)^{4^n} \, \big|\, \bx_e\in\mathbf{V}_e \big]^{1/4^n}\mik 1
\end{equation}
and define
\begin{equation} \label{e4.11}
\boldsymbol{\lambda}_e=1 \ \text{ and } \ \boldsymbol{\nu}_e=1+\boldsymbol{\varphi}_e.
\end{equation}
By Theorem \ref{t4.2}, H\"{o}lder's inequality and arguing precisely as in the previous example, we see that the weighted
hypergraph $\langle \nu_e:e\in\calh\rangle$ is a $(4^{4n},\eta,4)$-pseudorandom family for every $\eta>0$. Moreover,
a straightforward modification of the argument in the previous example shows that there exist weighted hypergraphs
of this form which cannot be realized as dense subhypergraphs of weighted hypergraphs satisfying the ``linear forms condition".

\subsection{The second main result}

Our second main result provides a somewhat different type of examples of pseudorandom families.
\begin{thm} \label{t4.3}
Let $n\in\nn$ with $n\meg 3$, $C\meg 1$ and $1<p\mik\infty$, and let $\ell$ be as in \eqref{e4.3}. Also let $0<\eta\mik 1/(n\ell)$
and let $\mathscr{H}=(n,\langle(X_i,\Sigma_i,\mu_i):i\in [n]\rangle,\calh)$ be a hypergraph system with $\calh=K^{(n-1)}_n={[n]\choose n-1}$.
$($Again observe that $\mathscr{H}$ is $(n-1)$-uniform.$)$ For every $e\in\calh$ let $\nu_e\in L_1(\bbx,\calb_e,\bmu)$ and
$\psi_e\in L_p(\bbx,\calb_e,\bmu)$ be nonnegative random variables, and let $\boldsymbol{\nu}_e$ and $\boldsymbol{\psi}_e$
be as in \eqref{e1.5} for $\nu_e$ and $\psi_e$ respectively. Assume that the following conditions are satisfied.
\begin{enumerate}
\item[(I)] We have
\begin{equation} \label{e4.12}
1-\eta \mik \ave\Big[ \prod_{e \in \calh} \ \prod_{\omega \in \{0,\ldots,\ell -1 \}^e }\!\! \boldsymbol{\psi}_e ^{n_{e,\omega}}
(\bx_e^{(\omega)})\, \Big|\, \bx^{(0)},\dots,\bx^{(\ell-1)}\in \boldsymbol{X}\Big]\mik C + \eta
\end{equation}
for any choice of $n_{e,\omega}\in \{0,1\}$.
\item[(II)] For every $e\in\calh$ we have $1\mik\|\boldsymbol{\nu}_e\|_{\square_{\ell,p}(\bbx_e)}<\infty$,
$\|\boldsymbol{\psi}_e\|_{\square_{\ell,p}(\bbx_e)}\mik C$ and
\begin{equation} \label{e4.13}
\|\boldsymbol{\nu}_e-\boldsymbol{\psi}_e\|_{\square_\ell(\bbx_e)} \mik \eta\,(C\cdot M)^{-(n-1)\ell}
\end{equation}
where $M=\max\{\|\boldsymbol{\nu}_e\|_{\square_{\ell,p}(\bbx_e)}:e\in\calh\}$.
\end{enumerate}
Then the family $\langle \nu_e \colon e \in \calh\rangle$ is $(C,\eta',p)$-pseudorandom where $\eta'=n\ell\eta$.
\end{thm}
Notice that in Theorem \ref{t4.3} each $\nu_e$ is decomposed as $\psi_e+(\nu_e-\psi_e)$. Here, the condition on the first
components---that is, condition \eqref{e4.12}---is weaker than that in Theorem \ref{t4.2}, but this is offset by making stronger
the condition on the pseudorandom components. We also remark that Theorem \ref{t4.3} was motivated by \cite[Lemmas 5 and 6]{CFZ2}
which dealt\footnote{We notice that if $\psi_e=1$ for every $e\in\calh$, then a slight weakening of \eqref{e4.13} is only needed.
Specifically, one can assume that $\|\boldsymbol{\nu}_e-1\|_{\square_2(\bbx_e)} \mik \eta M^{-(n-1)}$ where
$M=\max\{\|\boldsymbol{\nu}_e\|_{L_\infty}:e\in\calh\}$; see \cite{CFZ2} for details.} with the case $C=1$, $p=\infty$ and $\psi_e=1$
for every~$e\in\calh$. Its proof is given in Section 6.


\section{Proof of Theorem \ref{t4.2}}

\numberwithin{equation}{section}

Let $n, C, p, \ell$ and $\eta$ be as in the statement of the theorem and set
\begin{equation} \label{e5.1}
\bar{C}=(2C)^{n\ell}, \ \ \bar{\eta}=(2C)^{n\ell}\,\eta^{1/\ell^{n-1}}, \ \ C'=(4C)^{n\ell} \ \text{ and } \ \eta' = (4C)^{n\ell}\eta^{1/\ell^{n-1}}.
\end{equation}
Also let $\mathscr{H}=(n,\langle(X_i,\Sigma_i,\mu_i):i\in [n]\rangle,\calh)$ be a hypergraph system with $\calh\!=\!{n\choose n-1}$ and
for every $e\in\calh$ let $\lambda_e\in L_1(\bbx,\calb_e,\bmu)$ and $\varphi_e\in L_p(\bbx,\calb_e,\bmu)$ be nonnegative random variables
satisfying (I) and (II).

We will need the following lemma. Its proof is given in Subsection 5.1.
\begin{lem} \label{l5.1}
Let $e\in\calh$ and let $i\in [n]$ be the unique integer such that $e=[n]\setminus\{i\}$. For every $e'\in\calh\setminus\{e\}$ and every
$\omega\in\{0,\dots,\ell-1\}$ let $g_{e'}^{(\omega)}\in L_1(\bbx,\calb_{e'},\bmu)$ such that either: \emph{(i)} $0\mik g_{e'}^{(\omega)}\mik \lambda_{e'}$,
or \emph{(ii)} $0 \mik g_{e'}^{(\omega)} \mik \varphi_{e'}$, or \emph{(iii)} $0 \mik g_{e'}^{(\omega)} \mik 1$.~Then~we~have
\begin{equation} \label{e5.2}
\Big|\ave\Big[ (\boldsymbol{\lambda}_{e}-1)(\bx_{e}) \prod_{\omega=0}^{\ell-1} \prod_{e'\in\calh\setminus\{e\}\!\!}
g_{e'}^{(\omega)}(\bx_{e}, x_{i}^{(\omega)})\, \Big|\, \bx_{e}\in \bbx_{e},x_i^{(0)}, x_i^{(1)}\in X_i\Big]\Big|\mik \bar{\eta}
\end{equation}
and
\begin{equation} \label{e5.3}
\ave\Big[ \prod_{\omega=0}^{\ell-1} \prod_{e'\in\calh\setminus\{e\}}\!\! g_{e'}^{(\omega)}(\bx_{e},x^{(\omega)}_{i})\, \Big|\,
\bx_{e}\in \bbx_{e},x_i^{(0)},\dots, x_i^{(\ell-1)}\in X_i \Big] \mik \bar{C}
\end{equation}
where $\bar{\eta}$ and $\bar{C}$ are as in \eqref{e5.1}.
\end{lem}
After this preliminary discussion we are ready to enter into the main part of the proof. For every $e\in\calh$ set $\nu_e=\lambda_e+\varphi_e$
and let $\boldsymbol{\nu}_e$ be as in \eqref{e1.5} for $\nu_e$. Recall that we need to verify conditions (C1)--(C3)
in Definition \ref{d4.1} for the family $\langle \nu_e:e\in\calh\rangle$.

First, let $\calg\subseteq\calh$ be nonempty. Since $0\mik \lambda_e\mik \nu_e$, by \eqref{e4.4}, we have
\[ 1-\eta \mik \ave\Big[ \prod_{e \in \calg} \lambda_e\Big] \mik \ave\Big[\prod_{e \in \calg} \nu_e\Big] \]
and so, condition (C1) is satisfied.

Next, for every $e\in\calh$ let $\psi_e= \varphi_e+1$. By part (c.i) of Proposition \ref{p2.3},
\[ \|\psi_e\|_{L_p}\mik \|\varphi_e\|_{L_p}+1\mik \|\boldsymbol{\varphi}_e\|_{\square_{\ell, p}(\bbx_e)}+1\mik  C+1 \stackrel{\eqref{e5.1}}{\mik} C'.\]
Fix $e\in\calh$ and for every $f\in\partial e$ let $A_f\in\calb_f$. For every $e'\in\calh\setminus\{e\}$ and every $\omega\in\{0,\dots,\ell-1\}$
we define $g_{e'}^{(\omega)}\in L_1(\bbx,\calb_{e'},\bmu)$ by setting $g_{e'}^{(\omega)}=1$ if $\omega\in \{1,\dots,\ell-1\}$
and $g_{e'}^{(0)}=\mathbf{1}_{A_{f}}$ where $f=e'\cap e$. By \eqref{e5.2}, we have
\begin{equation} \label{e5.4}
\Big|\ave\Big[(\boldsymbol{\lambda}_{e}-1)(\bx_{e}) \prod_{f\in\partial e}\mathbf{1}_{A_{f}}(\bx_{e}, x_{i})\, \Big|\,
\bx_{e}\in \bbx_{e},x_i^{(0)}, x_i^{(1)}\in X_i \Big]\Big|\mik \bar{\eta}.
\end{equation}
Hence, by \eqref{e5.4}, the definition of the cut norm and the fact that $\nu_e-\psi_e=\lambda_e-1$, we conclude that
$\|\nu_e-\psi_e\|_{\cals_{\partial e}}\mik \eta'$. That is, condition (C2.a) is satisfied.

We proceed to verify condition (C2.b). Let $e\in\calh$ be arbitrary and let $i\in [n]$ be the unique integer such that $e=[n]\setminus\{i\}$.
Also let $\omega \in \{0,1\}$. For every $e'\in\calh\setminus\{e\}$ let $g_{e'}^{(\omega)}\in L_1(\bbx,\calb_{e'},\bmu)$ such that either
$0 \mik g_{e'}^{(\omega)}\mik\nu_{e'}$ or $0 \mik g_{e'}^{(\omega)} \mik 1$. We set
\[ \calg_{e,\nu}^{(\omega)}=\big\{e'\in \calh\setminus\{e\}: 0\mik g^{(\omega)}_{e'}\mik \nu_{e'}\big\} \ \text{ and } \
\calg_{e,1}^{(\omega)}=\big\{e'\in \calh\setminus\{e\}: 0\mik g^{(\omega)}_{e'}\mik 1\big\} \]
and for every $e'\in \calg_{e,\nu}^{(\omega)}$ let
\begin{equation} \label{e5.5}
g_{e',\lambda}^{(\omega)}=g_{e'}^{(\omega)}\,\mathbf{1}_{[g_{e'}^{(\omega)}\mik \lambda_{e'}]} \ \text{ and } \
g_{e',\varphi}^{(\omega)}=(g_{e'}^{(\omega)}-\lambda_{e'})\,\mathbf{1}_{[g_{e'}^{(\omega)}>\lambda_{e'}]}.
\end{equation}
Finally, for every  $\calg\subseteq \calg_{e,\nu}^{(\omega)}$ set
\begin{equation} \label{e5.6}
A_{\calg}^{(\omega)}=\prod_{e'\in\calg}g_{e', \lambda}^{(\omega)} \prod_{e'\in\calg_{e,\nu}^{(\omega)}\setminus\calg}g_{e',\varphi}^{(\omega)}
\prod_{e'\in \calg_{e,1}^{(\omega)}}g_{e'}^{(\omega)}.
\end{equation}
(Recall that, by convention, the product of an empty family of functions is equal to the constant function 1.)
The following properties are straightforward consequences of the relevant definitions.
\begin{enumerate}
\item[(a)] For every $\omega\in\{0,1\}$ and every $e'\in\calg_{e,\nu}^{(\omega)}$ we have
$g_{e',\lambda}^{(\omega)},\, g_{e',\varphi}^{(\omega)}\in L_1(\bbx,\calb_{e'},\bmu)$, $0\mik g_{e',\lambda}^{(\omega)} \mik \lambda_{e'}$,
$0\mik g_{e',\varphi}^{(\omega)} \mik \varphi_{e'}$ and $g_{e'}^{(\omega)}=g_{e',\lambda}^{(\omega)}+g_{e',\varphi}^{(\omega)}$.
\item[(b)] For every $\bx_{e}\in\bbx_{e}$ and every $x_i^{(0)}, x_i^{(1)}\in X_i$ we have
\begin{equation} \label{e5.7}
\prod_{\omega\in\{0,1\}}\prod_{e'\in\calh\setminus\{e\}} g_{e'}^{(\omega)}(\bx_{e}, x_i^{(\omega)}) = \sum_{\calg_0\subseteq \calg_{e,\nu}^{(0)}}\,
\sum_{\calg_1\subseteq \calg_{e,\nu}^{(1)}}\, \prod_{\omega\in\{0,1\}} A^{(\omega)}_{\calg_{\omega}}(\bx_{e},x_i^{(\omega)}).
\end{equation}
\end{enumerate}
By (a), we see that every factor of $A_{\calg}^{(\omega)}$ satisfies the assumptions of Lemma \ref{l5.1}. Therefore,
\[ \begin{split}
& \ave \Big[ (\boldsymbol{\nu}_{e}-\boldsymbol{\psi}_{e})(\bx_e) \prod_{\omega\in\{0,1\}} \ave\big[ \prod_{e'\in\calh\setminus\{e\}}
g_{e'}^{(\omega)}(\bx_{e},x_i) \, \big| \, x_i\in X_i\big]\,\Big|\,\bx_{e}\in \bbx_{e}\Big] \\
& \stackrel{ \ \ \ \ \,\,}{=} \,\ave\Big[ (\boldsymbol{\lambda}_{e}-1)(\bx_{e}) \prod_{\omega\in\{0,1\}} \prod_{e'\in\calh\setminus\{e\}}\!
g_{e'}^{(\omega)}(\bx_{e},x_i^{(\omega)})\, \Big|\, \bx_{e}\in\bbx_{e},x_i^{(0)},x_i^{(1)}\in X_i \Big]\\
& \stackrel{\eqref{e5.7}}{=} \sum_{\calg_0\subseteq\calg_{e,\nu}^{(0)}}\, \sum_{\calg_1\subseteq \calg_{e,\nu}^{(1)}}
\ave\Big[ (\boldsymbol{\lambda}_{e}-1)(\bx_{e}) \prod_{\omega\in\{0,1\}}\! A_{\calg_\omega}^{(\omega)}(\bx_{e}, x_i^{(\omega)})
\,\Big|\, \bx_{e}\in\bbx_{e}, x_i^{(0)}, x_i^{(1)}\in X_i \Big] \\
& \stackrel{\eqref{e5.2}}{\mik} \, 2^{|\calg_{e,\nu}^{(0)}|}\, 2^{|\calg_{e,\nu}^{(1)}|}\,\bar{\eta}\, \mik\, 4^{n-1}\bar{\eta}
\stackrel{\eqref{e5.1}}{\mik} \eta'
\end{split} \]
which implies, of course, that condition (C2.b) is satisfied.

It remains to verify condition (C3). Fix $e\in\calh$ and, as above, let $i\in [n]$ be the unique integer such that $e=[n]\setminus\{i\}$.
Also let $\calg\subseteq\calh\setminus \{e\}$ be nonempty and let $\boldsymbol{\nu}_{e,\calg}\colon \bbx_e\to\rr$ be as in Definition \ref{d4.1}.
Then notice that
\begin{eqnarray*}
\ave[\boldsymbol{\nu}_{e,\calg}^{\ell}] & = &
\ave\Big[ \ave\big[ \prod_{e'\in\calg}\nu_{e'}(\bx_e, x_i) \,\big|\, x_i\in X_i\big]^\ell\,\Big|\,\bx_e\in\bbx_e\Big] \\
& = & \ave\Big[ \prod_{\omega=0}^{\ell-1}\prod_{e'\in\calg}\nu_{e'}(\bx_e, x_i^{(\omega)}) \, \Big| \,
\bx_e\in\bbx_e,x_i^{(0)},\dots, x_i^{(\ell-1)}\in X_i \Big].
\end{eqnarray*}
Next, observe that for every $\bx_{e}\in\bbx_{e}$ and every $x_i^{(0)}, \dots, x_i^{(\ell-1)}\in X_i$ we have
\begin{multline*}
\ \ \ \ \ \prod_{\omega=0}^{\ell-1}\prod_{e'\in\calg}\nu_{e'}(\bx_e, x_i^{(\omega)}) =
\prod_{\omega=0}^{\ell-1} \prod_{e'\in\calg} (\lambda_{e'}+\varphi_{e'})(\bx_e, x_i^{(\omega)}) = \\
= \sum_{\calg_0\subseteq \calg} \cdots \sum_{\calg_{\ell-1}\subseteq \calg}\, \prod_{\omega=0}^{\ell-1} \prod_{e'\in\calg_\omega}\lambda_{e'}
(\bx_e, x_i^{(\omega)})\!\prod_{e'\in\calg\setminus \calg_\omega}\varphi_{e'}(\bx_e, x_i^{(\omega)}). \ \ \ \ \ \ \
\end{multline*}
Therefore, setting
\[ B_{\calg'}=\prod_{e'\in\calg'}\lambda_{e'} \prod_{e'\in\calg\setminus \calg'}\varphi_{e'} \]
for every $\calg'\subseteq \calg$, we obtain that
\begin{eqnarray*}
\ave[\boldsymbol{\nu}_{e,\calg}^{\ell}]\!\! & = & \!\!
\sum_{\calg_0\subseteq \calg} \cdots \sum_{\calg_{\ell-1}\subseteq \calg} \ave\Big[\prod_{\omega=0}^{\ell-1}
B_{\calg_{\omega}}(\bx_e, x_i^{(\omega)}) \, \Big| \, \bx_e\in\bbx_e, x_i^{(0)},\dots, x_i^{(\ell-1)}\in X_i \Big] \\
& \stackrel{\eqref{e5.3}}{\mik} & 2^{|\calg|\ell} \bar{C} \, \mik \, 2^{(n-1)\ell} \bar{C} \stackrel{\eqref{e5.1}}{\mik} C'.
\end{eqnarray*}
This shows that condition (C3) is satisfied, and so the entire proof of Theorem \ref{t4.2} is completed.

\subsection{Proof of Lemma \ref{l5.1}}

The argument is similar to that in the proofs of \cite[~Lemma 6.3]{CFZ1} and \cite[Proposition 5.1]{Tao2}.
Proofs of this sort originate from the work of Green and Tao \cite{GT1,GT2}.

We proceed to the details. First we need to introduce some pieces of notation. Let $d$ be a (possibly empty) subset of $e$ and write
$\bbx~=~\bbx_{e\setminus d}~\times~\bbx_{d\cup\{i\}}$. (Recall that $i\in [n]$ is the unique integer such that $e=[n]\setminus\{i\}$.)
Notice that every element of the space $\bbx_{e\setminus d}\times \bbx_{d\cup\{i\}}^\ell$ is written as
$(\bx_{e\setminus d},\bx_{d\cup \{i\}}^{(0)},\dots, \bx_{d\cup\{i\}}^{(\ell-1)})$ where $\bx_{e\setminus d}~\in~\bbx_{e\setminus d}$ and
$\bx_{d\cup\{i\}}^{(0)},\dots,\bx_{d\cup\{i\}}^{(\ell-1)}\in \bbx_{d\cup\{i\}}$. On the other hand, for every
$\bx_{d\cup\{i\}}^{(0)}=(x_j^{(0)})_{j\in d\cup\{i\}},\dots, \bx_{d\cup\{i\}}^{(\ell-1)}=(x_j^{(\ell-1)})_{j\in d\cup\{i\}}\in \bbx_{d\cup\{i\}}$,
every $d'~\subseteq~d~\cup~\{i\}$ and every $\omega=(\omega_j)_{j\in d'}\in\{0,\dots, \ell-1\}^{d'}$ by $\bx_{d'}^{(\omega)}$ we shall denote
the unique element of $\bbx_{d'}$ defined by the rule
\begin{equation} \label{e5.8}
\bx_{d'}^{(\omega)}=(x_j^{(\omega_j)})_{j\in d'}.
\end{equation}

Next, for every $d\subseteq e$ we define $F_d, G_d\colon \bbx_{e\setminus d}\times \bbx_{d\cup\{i\}}^\ell\to \rr$
as follows. First, set
\begin{equation} \label{e5.9}
F_d(\bx_{e\setminus d},\bx_{d\cup\{i\}}^{(0)},\dots,\bx_{d\cup\{i\}}^{(\ell-1)})=\prod_{\omega\in\{0,\dots,\ell-1\}^d}\!\!
\big(\boldsymbol{\lambda}_{e}(\bx_{e\setminus d},\bx_d^{(\omega)})-1\big)
\end{equation}
and notice that $F_d$ does not depend on the value of $x_i^{(0)},\dots,x_i^{(\ell-1)}\in X_i$. The definition of the function $G_d$
is somewhat more involved. For every $e'~\in~\calh~\setminus~\{e\}$ and every $\omega~\in~\{0,\dots,\ell-1\}$ let $g_{e'}^{(\omega)}$
be as in the statement of the~lemma~and~let~$\mathbf{g}_{e'}^{(\omega)}$ be as in \eqref{e5.1} for $g_{e'}^{(\omega)}$. Given
$e'\in\calh\setminus\{e\}$ and $\omega\in\{0,\dots,\ell-1\}^{e'\cap(d\cup\{i\})}$ it is convenient to introduce an auxiliary
function $g_{e',d,\omega}\colon\bbx_{e\setminus d}\times \bbx_{d\cup\{i\}}^\ell\to \rr$ by defining
$g_{e',d,\omega}(\bx_{e\setminus d},\bx_{d\cup \{i\}}^{(0)},\dots,\bx_{d\cup \{i\}}^{(\ell-1)})$
according to the following cases
\begin{equation} \label{e5.10}
\begin{cases}
\mathbf{g}_{e'}^{(\omega_i)}(\bx_{e'\setminus (d\cup\{i\})},\bx_{d\cup\{i\}}^{(\omega)})
& \!\!\text{if } d\subseteq e',\\
\boldsymbol{\lambda}_{e'}(\bx_{e'\setminus (d\cup\{i\})},\bx_{e'\cap (d\cup\{i\})}^{(\omega)})
& \!\!\text{if } d\nsubseteq e' \text{ and } g_{e'}^{(\omega_i)}\mik\lambda_{e'}, \\
1 & \!\!\text{if } d\nsubseteq e' \text{ and either } g_{e'}^{(\omega_i)}\mik\varphi_{e'} \text{ or } g_{e'}^{(\omega_i)}\mik 1.
\end{cases}
\end{equation}
(Here, $\omega_i$ is the $i$-th coordinate of $\omega$. Moreover, $\bx_{e'\setminus (d\cup\{i\})}$ stands for the natural
projection of $\bx_{e\setminus d}$ into $\bbx_{e'\setminus (d\cup\{i\})}$; note that this projection is well-defined since
$e'~\setminus~(d\cup\{i\})\subseteq e\setminus d$.) We now define
\begin{equation} \label{e5.11}
G_d=\prod_{e'\in\calh\setminus\{e\}}\, \prod_{\omega\in \{0,\dots,\ell-1\}^{e'\cap(d\cup\{i\})}} g_{e',d,\omega}.
\end{equation}
Finally, we set
\begin{equation} \label{e5.12}
Q_d=\ave[F_d \, G_d] \ \text{ and } \ R_d=\ave[G_d].
\end{equation}
\begin{claim} \label{c5.2}
The following hold.
\begin{enumerate}
\item[(a)] We have that $Q_\emptyset$ and $R_\emptyset$ coincide with the quantities appearing in the left-hand side
of \eqref{e5.2} and \eqref{e5.3} respectively.
\item[(b)] We have $|Q_e|\mik 2\eta$ and $0\mik R_e\mik 1+\eta$.
\end{enumerate}
\end{claim}
\begin{proof}[Proof of Claim \emph{\ref{c5.2}}]
For part (a) it is enough to observe that
\[ F_\emptyset(\bx_e, x_i^{(0)},\dots, x_i^{(\ell-1)})=\boldsymbol{\lambda}_e(\bx_e)-1 \]
and
\begin{eqnarray*}
G_\emptyset(\bx_e, x_i^{(0)},\dots, x_i^{(\ell-1)}) & = & \prod_{e'\in\calh\setminus\{e\}} \prod_{\omega\in\{0,\dots,\ell-1\}^{\{i\}}}
\mathbf{g}_{e'}^{(\omega_i)}(\bx_{e'\setminus \{i\}}, x_i^{(\omega)}) \\
& = & \prod_{\omega=0}^{\ell-1} \prod_{e'\in\calh\setminus\{e\}}g_{e'}^{(\omega)}(\bx_{e}, x^{(\omega)}_{i}).
\end{eqnarray*}
For part (b) notice that
\[ Q_{e}= \ave\Big[\!\!\!\!\!\prod_{\omega\in\{0,\dots,\ell-1\}^{e}}\!\!\!\!\!\!\!\big(\mathbf{\lambda}_{e} (\bx_{e}^{(\omega)})-1\big)\!\!\!\!
\prod_{e' \in \calh\setminus\{e\}} \prod_{\omega\in\{0,\dots,\ell-1\}^{e'}} \!\!\!\!\!\!\!\!\!\boldsymbol{\lambda}_{e'}^{n_{e',\omega}}(\bx_e^{(\omega)})
\,\Big|\, \bx^{(0)},\dots,\bx^{(\ell-1)}\in\bbx\Big] \]
and
\[ R_{e}= \ave\Big[\prod_{e' \in \calh\setminus\{e\}}
\prod_{\omega\in\{0,\dots,\ell-1\}^{e'}}\!\!\!\!\!\boldsymbol{\lambda}_{e'}^{n_{e',\omega}}(\bx_e^{(\omega)})
\,\Big|\, \bx^{(0)},\dots,\bx^{(\ell-1)}\in\bbx\Big] \]
where for every $e'\in\calh\setminus\{e\}$ and every $\omega\in \{0,\dots,\ell-1\}^{e'}$ we have $n_{e',\omega}\in\{0,1\}$ and $n_{e',\omega}=1$
if and only if $g_{e'}^{(\omega_i)}\mik \lambda_{e'}$. Therefore, by \eqref{e4.4}, we conclude that $|Q_e|\mik 2\eta$ and $0\mik R_e\mik 1+\eta$.
\end{proof}
The following claim is the last step of the proof.
\begin{claim} \label{c5.3}
For every $d\varsubsetneq e$ and every  $j\in e\setminus d$ we have $Q_{d\cup \{j\}}\meg 0$, and
\begin{equation} \label{e5.13}
|Q_d|^{1/\ell^{|d|}} \mik (2C)^{\ell}\, Q_{d \cup \{j\}}^{1/\ell^{|d|+1}} \ \text{ and } \
R_d^{1/\ell^{|d|}} \mik (2C)^{\ell} \, R_{d \cup \{j\}}^{1/\ell^{|d|+1}}.
\end{equation}
\end{claim}
\begin{proof}[Proof of Claim \emph{\ref{c5.3}}]
We will only show that $|Q_d|^{1/\ell^{|d|}} \mik (2C)^{\ell}\,Q_{d \cup \{j\}}^{1/\ell^{|d|+1}}$. The proof of the corresponding inequality
for $R_d$ is identical. (In particular, it follows by setting $F_d=1$ below.)

Fix $d\varsubsetneq e$ and $j\in e\setminus d$, and set $e_j=[n]\setminus\{j\}$ and $f=[n]\setminus(d\cup\{i,j\})$. Also write
$\bbx_{e_j}=\bbx_{f}\times \bbx_{d\cup \{i\}}$ and $\bbx_{e\setminus d}\times \bbx_{d\cup\{i\}}^\ell=X_j\times\bbx_{f}\times \bbx_{d\cup\{i\}}^\ell$,
and let $\pi\colon \bbx_{e\setminus d}\times \bbx_{d\cup\{i\}}^\ell\to \bbx_{f}\times \bbx_{d\cup\{i\}}^\ell$ denote the natural projection.

For every $\omega\in\{0,\dots,\ell-1\}^{d\cup\{i\}}$ we define $\mathbf{g}_{e_j,d,\omega}\colon\bbx_{f}\times \bbx_{d\cup\{i\}}^\ell\to \rr$ by
\begin{equation} \label{e5.14}
\mathbf{g}_{e_j,d,\omega}(\bx_{f},\bx_{d\cup \{i\}}^{(0)},\dots,\bx_{d\cup \{i\}}^{(\ell-1)})=
\mathbf{g}_{e_j}^{(\omega_i)}(\bx_{f}, \bx_{d\cup\{i\}}^{(\omega)}).
\end{equation}
(Recall that $\omega_i$ is the $i$-th coordinate of $\omega$ and  $\mathbf{g}_{e_j}^{(\omega_i)}$ is as in \eqref{e5.1} for
$g_{e_j}^{(\omega_i)}$.) Observe that, by \eqref{e5.10}, for every $\omega\in \{0,\dots,\ell-1\}^{d\cup\{i\}}$ we have
\begin{equation} \label{e5.15}
g_{e_j,d,\omega}=\mathbf{g}_{e_j,d,\omega}\circ \pi.
\end{equation}
Next, let $\Omega_\lambda$, $\Omega_\varphi$ and $\Omega_1$ denote the subsets of $\{0,\dots,\ell-1\}^{d\cup\{i\}}$ consisting of all $\omega$
such that $g_{e_j}^{(\omega_i)}\mik \lambda_{e_j}$, $g_{e_j}^{(\omega_i)}\mik \varphi_{e_j}$ and  $g_{e_j}^{(\omega_i)}\mik 1$ respectively. Set
\begin{equation} \label{e5.16}
\mathbf{G}_{j,d,\lambda}=\prod_{\omega\in \Omega_\lambda}
\mathbf{g}_{e_j,d,\omega}, \ \ \mathbf{G}_{j,d,\varphi}=\prod_{\omega\in \Omega_\varphi}
\mathbf{g}_{e_j,d,\omega} \ \text{ and } \ \mathbf{G}_{j,d,1}=\prod_{\omega\in \Omega_1}
\mathbf{g}_{e_j,d,\omega}
\end{equation}
and notice that these functions are defined on $\bbx_{f}\times \bbx_{d\cup\{i\}}^\ell$. We also define
$G_{j,d}'\colon\bbx_{e\setminus d} \times \bbx_{d\cup\{i\}}^\ell\to \rr$ by the rule
\begin{equation} \label{e5.17}
G_{j,d}'=\prod_{e'\in\calh\setminus\{e,e_j\}}\prod_{\omega\in\{0,\dots,\ell-1\}^{e'\cap(d\cup\{i\})}}
g_{e',d,\omega}.
\end{equation}
By \eqref{e5.11} and \eqref{e5.15}--\eqref{e5.17}, we have
$G_d=G_{j,d}'\cdot \big[(\mathbf{G}_{j,d, \lambda}\,\mathbf{G}_{j,d,\varphi}\,\mathbf{G}_{j,d,1})\circ \pi\big]$ and so
\begin{equation} \label{e5.18}
Q_d=\ave\big[\ave[F_d \,G_{j,d}'\,|\, x_j\in X_j] \cdot (\mathbf{G}_{j,d,\lambda}\,\mathbf{G}_{j,d,\varphi}\,\mathbf{G}_{j,d,1})\big]
\end{equation}
where the outer expectation is over the space $\bbx_{f}\times \bbx_{d\cup\{i\}}^\ell$. Denote by $\ell'$ the conjugate exponent of $\ell$ and recall that
$\ell$ is an even positive integer. By \eqref{e5.18}, H\"older's inequality and the fact that $0\mik \mathbf{G}_{j,d,1}\mik 1$, we obtain that
\begin{eqnarray} \label{e5.19}
|Q_d| & = & \big|\ave \big[ \ave[F_d\,G_{j,d}'\,|\, x_j\in X_j] \cdot
\big( (\mathbf{G}_{j,d,\lambda}^{1/\ell}\,\mathbf{G}_{j,d,\lambda}^{1/\ell'})\, \mathbf{G}_{j,d,\varphi}\,\mathbf{G}_{j,d,1}\big)\big]\big| \\
& \mik & \ave\big[ \ave[F_d\,G_{j,d}' \, |\, x_j\in X_j ]^\ell \cdot \mathbf{G}_{j,d,\lambda}\big]^{1/\ell} \cdot
\ave[\mathbf{G}_{j,d,\lambda} \, \mathbf{G}_{j,d,\varphi}^{\ell'}]^{1/\ell'}. \nonumber
\end{eqnarray}
Now define $\overline{\mathbf{G}}_{j,d, \lambda}, \overline{\mathbf{G}}_{j,d,\varphi}\colon\bbx_{f}\times \bbx_{d\cup\{i\}}^\ell\to \rr$ by
\begin{equation} \label{e5.20}
\overline{\mathbf{G}}_{j,d,\lambda}(\bx_{f},\bx_{d\cup\{i\}}^{(0)},\dots,\bx_{d\cup\{i\}}^{(\ell-1)})=
\prod_{\omega\in\Omega_\lambda} \boldsymbol{\lambda}_{e_j}(\bx_{f},\bx_{d\cup\{i\}}^{(\omega)})
\end{equation}
and
\begin{equation} \label{e5.21}
\overline{\mathbf{G}}_{j,d,\varphi}(\bx_{f},\bx_{d\cup\{i\}}^{(0)},\dots,\bx_{d\cup\{i\}}^{(\ell-1)})=
\prod_{\omega\in\Omega_\varphi} \boldsymbol{\varphi}_{e_j}(\bx_{f},\bx_{d\cup\{i\}}^{(\omega)}).
\end{equation}
By \eqref{e5.14} and \eqref{e5.16}, we see that $0\mik \mathbf{G}_{j,d,\lambda}\mik \overline{\mathbf{G}}_{j,d,\lambda}$ and
$0\mik \mathbf{G}_{j,d,\varphi}\mik \overline{\mathbf{G}}_{j,d,\varphi}$. Therefore, by \eqref{e5.19}, we have
\begin{equation} \label{e5.22}
|Q_d|\mik \ave\big[\ave[F_d \,G_{j,d}'\,|\, x_j\in X_j]^\ell \cdot \overline{\mathbf{G}}_{j,d,\lambda}\big]^{1/\ell} \cdot
\ave[\overline{\mathbf{G}}_{j,d,\lambda}\,\overline{\mathbf{G}}_{j,d,\varphi}^{\ell'}]^{1/\ell'}.
\end{equation}
It is easy to see that
\begin{equation} \label{e5.23}
\ave\big[ \ave[F_d \,G_{j,d}'\,|\, x_j\in X_j]^\ell \cdot \overline{\mathbf{G}}_{j,d,\lambda}\big]=Q_{d\cup\{j\}}
\end{equation}
which implies, in particular, that $Q_{d\cup\{j\}}\meg 0$. On the other hand, by \eqref{e5.20}, \eqref{e5.21} and part (a)
of Proposition \ref{p2.1}, we obtain that
\begin{equation} \label{e5.24}
\ave[\overline{\mathbf{G}}_{j,d,\lambda}\, \overline{\mathbf{G}}_{j,d,\varphi}^{\ell'}] \mik
\|\boldsymbol{\lambda}_{e_j}\|^{|\Omega_\lambda|}_{\square_\ell(\bbx_{e_j})}\cdot
\|\boldsymbol{\varphi}_{e_j}^{\ell'}\|^{|\Omega_\varphi|}_{\square_\ell(\bbx_{e_j})}.
 \end{equation}
By \eqref{e4.4}, it is clear that $\|\boldsymbol{\lambda}_{e_j}\|_{\square_\ell(\bbx_{e_j})}\mik 1+\eta\mik 2$. Moreover, by \eqref{e4.3},
we see that $1<\ell'<p$. Hence, by part (c.ii) of Proposition \ref{p2.3} and condition (II) in Theorem~\ref{t4.2}, we have
$\|\boldsymbol{\varphi}_{e_j}^{\ell'}\|_{\square_\ell(\bbx_{e_j})}\mik
\|\boldsymbol{\varphi}_{e_j}\|_{\square_{\ell,p}(\bbx_{e_j})}^{\ell'}\mik C^{\ell'}$.
Thus, by \eqref{e5.24},
\begin{equation} \label{e5.25}
\ave[\overline{\mathbf{G}}_{j,d,\lambda}\,\overline{\mathbf{G}}_{j,d,\varphi}^{\ell'}]^{1/\ell'} \mik
2^{|\Omega_\lambda|/\ell'}\cdot C^{|\Omega_\varphi|} \mik (2C)^{|\Omega_\lambda|+|\Omega_\varphi|} \mik (2C)^{\ell^{|d|+1}}.
\end{equation}
Combining \eqref{e5.22}, \eqref{e5.23} and \eqref{e5.25}, we get that $|Q_d |\mik Q_{d\cup\{j\}}^{1/\ell} (2C)^{\ell^{|d|+1}}$
which is equivalent to saying that $|Q_d|^{1/\ell^{|d|}} \mik (2C)^{\ell} \, Q_{d \cup \{j\}}^{1/\ell^{|d|+1}}$.
The proof of Claim \ref{c5.3} is completed.
\end{proof}
By induction and using Claim \ref{c5.3}, we see that
\begin{equation} \label{e5.26}
|Q_\emptyset|\mik (2C)^{(n-1)\ell}\, Q_{e}^{1/\ell^{n-1}} \ \text{ and } \ R_\emptyset\mik (2C)^{(n-1)\ell}\, R_{e}^{1/\ell^{n-1}}.
\end{equation}
Invoking \eqref{e5.26} and Claim \ref{c5.2}, we conclude that \eqref{e5.2} and \eqref{e5.3} are satisfied, and so the proof of Lemma \ref{l5.1}
is completed.


\section{Proof of Theorem \ref{t4.3}}

\numberwithin{equation}{section}

Let $\mathscr{H}$ and $\langle\nu_e,\psi_e:e\in\calh\rangle$ be as in the statement of the theorem,
and for every $e\in\calh$ let $\boldsymbol{\nu}_e$ and $\boldsymbol{\psi}_e$ be as in \eqref{e1.5} for
$\nu_e$ and $\psi_e$ respectively.

The following lemma is the first main step of the proof.
\begin{lem} \label{l6.1}
Let $e\in\calh$ and let $i\in [n]$ be the unique integer such that $e=[n]\setminus\{i\}$. For every $e'\in\calh\setminus\{e\}$ and every
$\omega\in\{0,\dots,\ell-1\}$ let $g_{e'}^{(\omega)}\in L_1(\bbx,\calb_{e'},\bmu)$ such that either: \emph{(i)} $0\mik g_{e'}^{(\omega)}\mik \nu_{e'}$,
or \emph{(ii)} $0 \mik g_{e'}^{(\omega)} \mik \psi_{e'}$, or \emph{(iii)} $0 \mik g_{e'}^{(\omega)} \mik 1$. Then
\begin{equation} \label{e6.1}
\!\!\!\Big|\ave\Big[(\bnu_{e}-\boldsymbol{\psi}_e)(\bx_{e})\!\prod_{\omega=0}^{\ell-1}\prod_{e'\in\calh\setminus\{e\}}\!\!\!\!\!
g_{e'}^{(\omega)}(\bx_{e}, x_{i}^{(\omega)}) \Big|\, \bx_{e}\in\bbx_{e},x_i^{(0)},\dots, x_i^{(\ell-1)}\! \in X_i\Big] \Big|\!\mik\!\eta.
\end{equation}
\end{lem}
\begin{proof}
It is similar to the proof of Lemma \ref{l5.1}. Specifically, let $d\subseteq e$ be arbitrary. For every
$e'~\in~\calh~\setminus\{e\}$ with $d \subseteq e'$ and every $\omega \in \{0,\dots,\ell-1\}^{d \cup \{i\}}$ we define
$g_{e',d,\omega}\colon \bbx_{e\setminus d}\times \bbx^\ell_{d\cup\{i\}}\to \rr$ by setting
\begin{equation} \label{e6.2}
g_{e',d,\omega}(\bx_{e\setminus d},\bx_{d\cup \{i\}}^{(0)},\dots,\bx_{d\cup \{i\}}^{(\ell-1)})=
\mathbf{g}_{e'}^{(\omega_i)}(\bx_{e'\setminus (d\cup\{i\})}, \bx_{d\cup\{i\}}^{(\omega)})
\end{equation}
where $\omega_i$ is the $i$-th coordinate of $\omega$ and $\mathbf{g}_{e'}^{(\omega_i)}$ is as in \eqref{e1.5} for $g_{e'}^{(\omega_i)}$.
Also define $F_d,G_d\colon\bbx_{e\setminus d}\times \bbx^\ell_{d\cup\{i\}}\to \rr$ by
\begin{equation} \label{e6.3}
F_d(\bx_{e\setminus d},\bx_{d\cup\{i\}}^{(0)},\dots,\bx_{d\cup\{i\}}^{(\ell-1)})=
\prod_{\omega\in\{0,\dots,\ell-1\}^d} (\boldsymbol{\nu}_{e}-\boldsymbol{\psi}_e)(\bx_{e\setminus d},\bx_d^{(\omega)})
\end{equation}
and
\begin{equation} \label{e6.4}
G_d=\prod_{\substack{e'\in\calh\setminus\{e\}\\ d\, \subseteq e'}}\, \prod_{\omega \in \{0,\dots,\ell-1\}^{d \cup \{i\}}} g_{e',d,\omega}.
\end{equation}
(Here, as in Section 3, we follow the convention that the product of an empty family of functions is equal to the constant function 1.) Finally, let
\begin{equation} \label{e6.5}
Q_d=\ave[F_d \, G_d].
\end{equation}
Note that $Q_\emptyset$ coincides with the quantity appearing in the left-hand side of \eqref{e6.1}. Moreover, we have
\begin{equation} \label{e6.6}
Q_{e}= \|\bnu_e-\boldsymbol{\psi}_e\|^{\ell^{n-1}}_{\square_\ell(\bbx_e)}.
\end{equation}
\begin{claim} \label{c6.2}
For every $d\varsubsetneq e$ and every $j\in e\setminus d$ we have $Q_{d\cup \{j\}}\meg 0$ and
\begin{equation} \label{e6.7}
|Q_d|^{1/\ell^{|d|}} \mik (C\cdot M)^{\ell} \cdot Q_{d \cup \{j\}}^{1/\ell^{|d|+1}}.
\end{equation}
\end{claim}
\begin{proof}[Proof of Claim \emph{\ref{c6.2}}]
As in the proof of Claim \ref{c5.3}, fix $d\varsubsetneq e$ and $j\in e\setminus d$,
and set $e_j=[n]\setminus\{j\}$ and $f=[n]\setminus (d\cup\{i,j\})$. Write $\bbx_{e_j}=\bbx_f\times\bbx_{d\cup \{i\}}$
and $\bbx_{e\setminus d}\times\bbx_{d\cup\{i\}}^\ell=X_j\times(\bbx_{f}\times \bbx_{d\cup\{i\}}^\ell)$, and for every
$\omega\in\{0,\dots,\ell-1\}^{d\cup\{i\}}$ define $\mathbf{g}_{e_j,d,\omega}\colon\bbx_{f}\times \bbx_{d\cup\{i\}}^\ell\to \rr$ by
\begin{equation} \label{e6.8}
\mathbf{g}_{e_j,d,\omega}(\bx_{f},\bx_{d\cup \{i\}}^{(0)},\dots,\bx_{d\cup \{i\}}^{(\ell-1)})=
\mathbf{g}_{e_j}^{(\omega_i)}(\bx_{f},\bx_{d\cup\{i\}}^{(\omega)}).
\end{equation}
If $\pi\colon\bbx_{e\setminus d}\times \bbx_{d\cup\{i\}}^\ell\to \bbx_{f}\times \bbx_{d\cup\{i\}}^\ell$ is the natural projection map, then,
by \eqref{e6.2}, we see that $g_{e_j,d,\omega}=\mathbf{g}_{e_j,d,\omega}\circ \pi$ for every $\omega\in \{0,\dots,\ell-1\}^{d\cup\{i\}}$.

Let $\Omega_\nu$, $\Omega_\psi$ and $\Omega_1$ denote the subsets of $\{0,\dots,\ell-1\}^{d\cup\{i\}}$ consisting of all $\omega$ such that
$g_{e_j}^{(\omega_i)}\mik \nu_{e_j}$, $g_{e_j}^{(\omega_i)}\mik \psi_{e_j}$ and $g_{e_j}^{(\omega_i)}\mik 1$ respectively. We set
\begin{equation} \label{e6.9}
\mathbf{G}_{j,d,\nu}=\prod_{\omega\in \Omega_\nu}\mathbf{g}_{e_j,d,\omega}, \ \
\mathbf{G}_{j,d,\psi}=\prod_{\omega\in \Omega_\psi}\mathbf{g}_{e_j,d,\omega} \ \text{ and } \
\mathbf{G}_{j,d,1}=\prod_{\omega\in \Omega_1}\mathbf{g}_{e_j,d,\omega}.
\end{equation}
Moreover, let $G_{j,d}'\colon\bbx_{e\setminus d}\times \bbx_{d\cup\{i\}}^\ell\to \rr$ be defined by
\begin{equation} \label{e6.10}
G_{j,d}'=\prod_{\substack{e'\in\calh\setminus\{e, e_j\}\\ d\,\subseteq e'}}\, \prod_{\omega\in\{0,\dots,\ell-1\}^{d\cup\{i\}}} g_{e',d,\omega}.
\end{equation}
Observe that $G_d=G_{j,d}'\cdot \big[ (\mathbf{G}_{j,d,\nu}\,\mathbf{G}_{j,d,\varphi}\,\mathbf{G}_{j,d,1})\circ \pi\big]$.
Hence, if $\ell'$ denotes the conjugate exponent of $\ell$, then, by H\"older's inequality, we have
\begin{eqnarray} \label{e6.11}
|Q_d| &= & \big|\ave\big[\ave[F_d \,G_{j,d}'\,|\, x_j\in X_j]\cdot (\mathbf{G}_{j,d,\nu}\,\mathbf{G}_{j,d,\psi}\,\mathbf{G}_{j,d,1}) \big]\big|\\
&\mik & \ave\big[ \ave[F_d \,G_{j,d}'\,|\, x_j\in X_j]^\ell\big]^{1/\ell} \cdot
\ave\big[ (\mathbf{G}_{j,d,\nu}\,\mathbf{G}_{j,d,\psi})^{\ell'}\big]^{1/ \ell'} \nonumber
\end{eqnarray}
where the outer expectation is over the space $\bbx_{f}\times \bbx_{d\cup\{i\}}^\ell$. Note that
\begin{equation} \label{e6.12}
\ave\big[\ave[F_d \,G_{j,d}'\,|\, x_j\in X_j]^\ell\big]=Q_{d\cup\{j\}}
\end{equation}
and, consequently, $Q_{d\cup\{j\}}\meg 0$. Next, define
$\overline{\mathbf{G}}_{j,d, \nu},\overline{\mathbf{G}}_{j,d,\psi}\colon \bbx_{f}\times \bbx_{d\cup\{i\}}^\ell\to \rr$ by
\begin{equation} \label{e6.13}
\overline{\mathbf{G}}_{j,d,\nu}(\bx_{f},\bx_{d\cup\{i\}}^{(0)},\dots,\bx_{d\cup\{i\}}^{(\ell-1)})=
\prod_{\omega\in\Omega_\nu} \boldsymbol{\nu}_{e_j}(\bx_{f},\bx_{d\cup\{i\}}^{(\omega)})
\end{equation}
and
\begin{equation} \label{e6.14}
\overline{\mathbf{G}}_{j,d,\psi}(\bx_{f},\bx_{d\cup\{i\}}^{(0)},\dots,\bx_{d\cup\{i\}}^{(\ell-1)})=
\prod_{\omega\in\Omega_\psi} \boldsymbol{\psi}_{e_j}(\bx_{f},\bx_{d\cup\{i\}}^{(\omega)}).
\end{equation}
By \eqref{e6.8} and \eqref{e6.9}, we see that $0\mik \mathbf{G}_{j,d,\nu}\mik \overline{\mathbf{G}}_{j,d,\nu}$ and
$0\mik \mathbf{G}_{j,d,\varphi}\mik \overline{\mathbf{G}}_{j,d,\varphi}$. On the other hand, by \eqref{e4.3}, we have $1<\ell'< p$.
Therefore, by \eqref{e6.13}, \eqref{e6.14} and parts~(a) and (c.ii) of Proposition \ref{p2.3}, we obtain that
\begin{eqnarray} \label{e6.15}
\ \ \ \ \ \ \ave\big[(\mathbf{G}_{j,d,\nu}\,\mathbf{G}_{j,d,\psi})^{\ell'}\big]^{1/\ell'}\!\!\!\! & \mik &
\!\!\!\ave\big[(\overline{\mathbf{G}}_{j,d,\nu}\,\overline{\mathbf{G}}_{j,d,\psi})^{\ell'}\big]^{1/\ell'} \\
& \mik & \!\!\! \|\boldsymbol{\nu}_{e_j}\|_{\square_{\ell,p}(\bbx_{e_j})}^{|\Omega_{\nu}|}  \cdot
\|\boldsymbol{\psi}_{e_j}\|_{\square_{\ell,p}(\bbx_{e_j})}^{|\Omega_{\psi}|}\mik (C\cdot M)^{\ell^{|d|+1}}. \nonumber
\end{eqnarray}
By \eqref{e6.11}, \eqref{e6.12} and \eqref{e6.15}, we see that $|Q_d |\mik Q_{d\cup\{j\}}^{1/\ell} (C\cdot M)^{\ell^{|d|+1}}$
and the proof of Claim \ref{c6.2} is completed.
\end{proof}
By the above claim, we have
\[ |Q_\emptyset|\mik (C\cdot M)^{(n-1)\ell}\, Q_{e}^{1/\ell^{n-1}}. \]
As we have noted, $Q_\emptyset$ coincides with the quantity appearing in the left-hand side of \eqref{e6.1}. Thus, combining the previous
estimate with \eqref{e4.13} and \eqref{e6.6}, we conclude that \eqref{e6.1} is satisfied, and so the proof of Lemma \ref{l5.1} is completed.
\end{proof}
We proceed with the following lemma which is the second main step of the proof.
\begin{lem} \label{l6.3}
Let $e\in\calh$ and let $i\in [n]$ be the unique integer such that $e~=~[n]~\setminus~\{i\}$. Also let $e'\in\calh\setminus\{e\}$
and $\omega' \in \{0,\dots,\ell-1\}$. For every $e''\in\calh\setminus\{e\}$ and every $\omega\in\{0,\dots,\ell-1\}$ with
$(e'',\omega)\neq (e',\omega')$ let $g_{e''}^{(\omega)}\in L_1(\boldsymbol{X},\calb_{e''},\bmu)$ such that either:
\emph{(i)} $0\mik g_{e''}^{(\omega)}\mik\nu_{e''}$, or \emph{(ii)} $0 \mik g_{e''}^{(\omega)} \mik \psi_{e''}$,
or \emph{(iii)} $0 \mik g_{e''}^{(\omega)} \mik 1$. Then
\begin{equation} \label{e6.16}
\Big| \ave \Big[(\bnu_{e'}-\bpsi_{e'})(\bx_{{e'}\setminus \{i\}},x_{i}^{(\omega')})\!\!
\prod_{\substack{{e''} \in \calh \setminus \{e\} \\ \omega \in \{0,\dots,\ell-1\} \\ ({e''},\omega) \neq ({e'},\omega')}}\!\!
\mathbf{g}_{{e''}}^{(\omega)}(\bx_{{e''}\setminus \{i\}},x_{i}^{(\omega)})\Big] \Big| \mik \eta
\end{equation}
where the expectation is over all $\bx_{e} \in \bbx_{e}$ and $x_{i}^{(0)},\dots,x_{i}^{(\ell-1)}\in X_{i}$. $($Here, $\bx_{e'\setminus \{i\}}$
and $\bx_{e''\setminus \{i\}}$ are the projections of\, $\bx_e$ into $\bbx_{e'\setminus \{i\}}$ and $\bbx_{e''\setminus \{i\}}$ respectively.$)$
\end{lem}
\begin{proof}
Without loss of generality, and to simplify the exposition, we will assume that $\omega'=0$. For every $d \subseteq e'\setminus\{i\}$ let
$F_d,G_d \colon \bbx_{e\setminus d} \times \bbx^{\ell}_{d \cup \{i\}} \to \rr$ be defined by
\[ F_d(\bx_{e \setminus d},\bx_{d\cup\{i\}}^{(0)},\dots,\bx_{d\cup\{i\}}^{(\ell-1)})=
\prod_{\omega \in \{0,\dots,\ell-1\}^d} (\bnu_{e'}-\bpsi_{e'})(\bx_{e' \setminus (d \cup \{i\})},\bx_d^{(\omega)},x_{i}^{(0)}) \]
and
\[ G_d=\prod_{(e'',\omega)\in \Gamma_d} g_{e'',d,\omega} \]
where: (a) the set $\Gamma_d$ consists of all pairs $(e'',\omega)\in \calh\setminus\{e\}\times \{0,\dots,\ell-1\}^{d\cup\{i\}}$ such that
$d\subseteq e''\setminus\{i\}$ and $(e'',\omega_i)\neq (e',0)$, and (b) for every $(e'',\omega)\in\Gamma_d$ the function
$g_{{{e''}},d,\omega}\colon \bbx_{e\setminus d} \times \bbx^{\ell}_{d \cup \{i\}} \to \rr$ is defined by
\[ g_{{{e''}},d,\omega}(\bx_{e \setminus d },\bx_{d\cup\{i\}}^{(0)},\dots,\bx_{d\cup\{i\}}^{(\ell-1)}) =
\mathbf{g}_{e''}^{(\omega_i)}(\bx_{{{e''}}\setminus (d \cup \{i\})},\bx_{d }^{(\omega_d)}, x_{i}^{(\omega_{i})}). \]
(As before, $\omega_i$ is the $i$-th coordinate of $\omega$ and $\mathbf{g}_{e''}^{(\omega_i)}$ is as in \eqref{e1.5}
for $g_{e''}^{(\omega_i)}$. Moreover, $\bx_{e'\setminus (d\cup\{i\})}$ and $\bx_{e''\setminus (d\cup\{i\})}$ are the projections
of $\bx_{e\setminus d}$ into $\bbx_{e'\setminus (d\cup\{i\})}$ and $\bbx_{e''\setminus (d\cup\{i\})}$ respectively.)

Next, setting $Q_d=\ave[F_d\, G_d]$ and arguing precisely as in the proof of Lemma \ref{l6.1}, we obtain that $Q_{d\cup \{j\}}\meg 0$ and
$|Q_d|^{1/\ell^{|d|}} \mik (C\cdot M)^{\ell}\, Q_{d \cup \{j\}}^{1/\ell^{|d|+1}}$ for every $d\varsubsetneq e'\setminus\{i\}$ and every
$j\in e'\setminus (d\cup\{i\})$. Therefore,
\begin{equation} \label{e6.17}
|Q_\emptyset|\mik (C\cdot M)^{(n-2)\ell}\, Q_{e'\setminus\{i\}}^{1/\ell^{n-2}}.
\end{equation}
Now observe that $|Q_{\emptyset}|$ coincides with the quantity appearing in the left-hand side of \eqref{e6.16}.
On the other hand, we have
\[ Q_{e' \setminus \{i\}}=\ave[\mathbf{F}_{e'\setminus\{i\}}\,\mathbf{G}_{e'\setminus\{i\}}\,|\,
\bx_{e'}^{(0)},\dots,\bx_{e'}^{(\ell-1)}\in\bbx_{e'}] \]
where
\[ \mathbf{F}_{e'\setminus\{i\}}(\bx_{e'}^{(0)},\dots,\bx_{e'}^{(\ell-1)})=\!\!
\prod_{\omega \in \{0,\dots,\ell-1\}^{e' \setminus \{i\}}}\!\!\!\! (\bnu_{e'} - \bpsi_{e'}) (\bx_{e' \setminus \{i\}}^{(\omega)},x_{i}^{(0)}) \]
and
\[ \mathbf{G}_{e'\setminus\{i\}}(\bx_{e'}^{(0)},\dots,\bx_{e'}^{(\ell-1)})=\!\!
\prod_{\omega\in\{0,\dots,\ell-1\}^{e'\setminus\{i\}}\times [\ell-1]^{\{i\}}} \!\!\!\!
\mathbf{g}_{{{e'}}}^{(\omega_{i})}(\bx_{e'\setminus\{i\} }^{(\omega_{e'\setminus\{i\}})},x_{i}^{(\omega_{i})}). \]
(The arguments of the functions in the definitions of $\mathbf{F}_{e'\setminus\{i\}}$ and $\mathbf{G}_{e'\setminus\{i\}}$
follow from our previous conventions, mutatis mutandis.) Thus, by part (a) of Proposition \ref{p2.1}, we obtain that
\[ Q_{e' \setminus \{i\}}\mik \|\bnu_{e'}-\bpsi_{e'}\|_{\square_{\ell}(\bbx_{e'})}^{\ell^{n-2}} \cdot \,
\prod_{\omega=1}^{\ell-1} \|\mathbf{g}_{e'}^{(\omega)}\|_{\square_{\ell}(\bbx_{e'})}^{\ell^{n-2}} \]
and consequently, by \eqref{e6.17},
\begin{equation} \label{e6.18}
|Q_\emptyset|\mik (C\cdot M)^{(n-2)\ell} \cdot \|\bnu_{e'}-\bpsi_{e'}\|_{\square_{\ell}(\bbx_{e'})} \cdot \,
\prod_{\omega=1}^{\ell-1} \|\mathbf{g}_{e'}^{(\omega)}\|_{\square_{\ell}(\bbx_{e'})}.
\end{equation}
By part (c.ii) of Proposition \ref{p2.3}, for every $\omega\in [\ell-1]$ we have
\[ \|\mathbf{g}_{e'}^{(\omega)}\|_{\square_{\ell}(\bbx_{e'})}\mik \|\mathbf{g}_{e'}^{(\omega)}\|_{\square_{\ell, 1}(\bbx_{e'})}
\mik \|\mathbf{g}_{e'}^{(\omega)}\|_{\square_{\ell,p}(\bbx_{e'})}. \]
Hence, by \eqref{e6.18} and condition (II),
\begin{eqnarray*}
|Q_\emptyset| & \mik & (C\cdot M)^{(n-2)\ell} \cdot \|\bnu_{e'}-\bpsi_{e'}\|_{\square_{\ell}(\bbx_{e'})}\cdot (C\cdot M)^{\ell-1} \\
& \mik & (C\cdot M)^{(n-1)\ell}\cdot \|\bnu_{e'} - \bpsi_{e'}\|_{\square_{\ell}(\bbx_{e'})} \mik \eta
\end{eqnarray*}
and the proof of Lemma \ref{l6.3} is completed.
\end{proof}
We are now in a position to complete the proof of the theorem. Recall that we need to show that the family $\langle \nu_e:e\in\calh\rangle$
satisfies conditions (C1)--(C3) in Definition \ref{d4.1} for the constants $C$ and $\eta'= n\ell\eta$. For condition (C1) let
$\calg\subseteq \calh\setminus\{e\}$ be nonempty. Set $m=|\calg|$ and let $e'_1,\dots, e'_m$ be an enumeration of $\calg$. Notice that
\begin{eqnarray} \label{e6.19}
\,\, \Big|\ave\Big[\prod_{e'\in\calg} \nu_{e'}\Big]-\ave\Big[\prod_{e'\in\calg}\psi_{e'}\Big]\Big|\!\!\! & \mik &
\!\sum_{j=1}^{m}\, \Big|\ave\Big[\prod_{k<j} \psi_{e'_k} (\nu_{e'_j}-\psi_{e'_j}) \prod_{k>j} \nu_{e'_k}\Big]\Big| \\
& \stackrel{\eqref{e6.1}}{\mik} & \!(n-1)\,\eta \nonumber
\end{eqnarray}
and so, by condition (I), we obtain that
\[ \ave\Big[\prod_{e' \in \calg} \nu_{e'}\Big]\meg \ave\Big[\prod_{e' \in \calg } \psi_{e'}\Big]-(n-1)\cdot\eta\meg 1-\eta'.\]
That is, condition (C1) is satisfied. Condition (C2.a) follows arguing precisely as in the proof of Theorem \ref{t4.2},
while condition (C2.b) is an immediate consequence of \eqref{e6.1}. Finally, for condition (C3) fix $e\in\calh$ and let $i\in [n]$
be the unique integer such that $e=[n]\setminus\{i\}\in\calh$. Also let $\calg\subseteq\calh\setminus \{e\}$ be nonempty.
By the choice of~$\eta'$, it is enough to show that
\begin{equation} \label{e6.20}
\ave[\boldsymbol{\nu}_{e,\calg}^{\ell}]\mik C+ (|\calg|\cdot \ell +1)\,\eta.
 \end{equation}
To this end, set
\[ \Delta\coloneqq \Big|\ave\Big[\prod_{\omega=0}^{\ell-1} \prod_{e' \in \calg} \bnu_{e'}(\bx_{e' \setminus \{i\}},x_i^{(\omega)})\Big]
-\ave\Big[ \prod_{\omega=0}^{\ell-1} \prod_{e' \in \calg} \bpsi_{e'}(\bx_{e' \setminus \{i\}},x_i^{(\omega)})\Big]\Big| \]
(both expectations are over the space $\bbx_e\times X_i^\ell$) and note that, by condition (I),
\[ \ave[\boldsymbol{\nu}_{e,\calg}^{\ell}]\mik\Delta+C+\eta. \]
Next, by enumerating the set $\calg\times \{0,\dots,\ell-1\}$ and applying a telescoping argument as in \eqref{e6.19},
we see that $\Delta$ is bounded by a sum of $|\calg|\cdot \ell$ terms each of which has the form of the quantity appearing
in the left-hand side of \eqref{e6.16}. Therefore, by Lemma~\ref{l6.3}, we conclude that \eqref{e6.20} is satisfied, and so
the entire proof of Theorem~\ref{t4.3} is completed.


\end{document}